\documentclass[12pt]{amsart}
\usepackage{amsfonts, amssymb}
\usepackage{amsmath}
\usepackage{amsthm}
\usepackage{cite}
\numberwithin{equation}{section}

\newtheorem{theorem}{Theorem}[section]
\newtheorem{lemma}[theorem]{Lemma}

\begin{document}

\title[First order differential equation]{On the modulus of solutions of a first order differential equation}

\author{Yueyang Zhang}
\address{School of Mathematics and Physics, University of Science and Technology Beijing, No.~30 Xueyuan Road, Haidian, Beijing, 100083, People's Republic of China}
\email{zhangyueyang@ustb.edu.cn}
\thanks{The author is supported by a Project supported by the National Natural Science Foundation of China~{(12301091)} and the Ministry of Science and Technology of the People's Republic of China~{(G2021105019L)}. The author thanks professor Rod Halburd of the University College London for having a lot of valuable discussions on Hayman's equation.}



\date{\today}

\commby{}

\begin{abstract}
Let $P(z)=z^{n}+a_{n-2}z^{n-2}+\cdots+a_0$ be a nonconstant polynomial and $S(z)$ be a nonzero rational function and denote $h(z)=S(z)e^{P(z)}$. Let $\theta\in(0,\pi/2n)$ be a constant and $\varepsilon>0$ be a small constant. It is shown that if $f(z)$ is a solution of the first order differential equation $f'(z)=h(z)f(z)+1$, then there is a sequence $\{r_{k}\}$ such that the set $E=\cup_{l=0}^{\infty}[r_{2l},r_{2l+1}]$ has infinite logarithmic measure and for all $r\in E$,
\begin{equation}\tag{\dag}
\begin{split}
|f(re^{i\theta})|\geq (1-\varepsilon)\frac{\sqrt[n]{\sin n\theta}}{n}r\exp\left(e^{(1-\varepsilon)r^n\cos n\theta}\sin\varepsilon\right).
\end{split}
\end{equation}
When $h(z)=e^{z}$, we also give a lower bound for $|f(re^{i\theta})|$ for other values of $r$. The estimate in $(\dag)$ yields that the hyper-order $\varsigma(f)$ of $f(z)$ is equal to $n$, giving a partial answer to Br\"{u}ck's conjecture in uniqueness theory of meromorphic functions. An extension of the method also yields a complete description on the order of growth of entire solutions of a second order algebraic differential equation of Hayman in the autonomous case.

\end{abstract}

\maketitle


\section{Introduction}\label{Introduction}

Let $P(z)$ be a polynomial and $S(z)$ be a nonzero rational function in the finite complex plane $\mathbb{C}$. Denote $h(z)=S(z)e^{P(z)}$. In this paper, we are concerned with the modulus of the solutions of the first order differential equation
\begin{equation}\label{Eq 1}
\begin{split}
f'(z)=h(z)f(z)+1,
\end{split}
\end{equation}
which appears in the literature as a modification of the first order differential equation
\begin{equation}\label{classNequatoin1}
\begin{split}
g'(z)=h(z)\left(g(z)-z\right).
\end{split}
\end{equation}
Equation \eqref{classNequatoin1} becomes equation \eqref{Eq 1} if we let $f(z)=z-g(z)$. Equation \eqref{classNequatoin1} arises when applying Newton's method, namely $g(z)=z-H(z)/H'(z)$, to find zeros of entire functions like $H(z)=\int_{0}^{z}\exp(\int_{0}^{t}P_0(s)e^{P(s)}ds)dt$, where $P_0(z)$ is a polynomial. The set of all finite order meromorphic functions satisfying equations of the form in \eqref{classNequatoin1} is the so-called \emph{class N} in complex dynamics~\cite{Baker1984,Bergweiler1992,Bergweiler1993}. In this case, $P(z)$ is a constant~\cite{zhang2021-1}. Equation \eqref{Eq 1} also appears in uniqueness theory of meromorphic functions~\cite{Bruck1996,GundersenYang1998,YangYi2003} and can be a particular reduction of a second order algebraic differential equation of Hayman~\cite{Hayman1996,ChiangHalburd2003,HalburdWang,zhang2017,Zhang2024}. By analyzing the modulus of the solution $f(z)$ of equation \eqref{Eq 1}, we are able to provide answers to related problems in these two areas.

From now on we assume that $P(z)$ is nonconstant and has degree $n\geq 1$. By re-scaling and translating the variable $z$ suitably, we shall assume that $P(z)$ takes the form
\begin{equation}\label{Eq 41}
\begin{split}
P(z)=z^{n}+a_{n-2}z^{n-2}+\cdots+a_0,
\end{split}
\end{equation}
where $a_{n-2},\cdots, a_0$ are constants. When $n=1$, this means $P(z)=z$. We also assume that $S(z)=nz^{m}[1+O(z^{-1})]$ as $z\to\infty$ for some integer $m\in \mathbb{Z}$. All poles of $S(z)$ are located in a finite open disk $D(0,\dot{r})=\{z=x+iy:|z|<\dot{r}\}$ for some $\dot{r}>0$. Then the solution $f(z)$ of equation \eqref{Eq 1} is analytic outside $D(0,\dot{r})$. In particular, if $S(z)$ is a polynomial, then the general solutions of equation \eqref{Eq 1} are entire functions.

Throughout the paper, $\varepsilon>0$ denotes a small constant. Moreover, the notation $O(1)$ means the related quantity is bounded while the notation $o(1)$ means the related quantity is infinitely small respectively. The main purpose of this paper is to provide a lower bound for $|f(z)|$ along certain rays. We first prove the following

\begin{theorem}\label{maintheorem1}
Let $P(z)$ be a polynomial of the form in \eqref{Eq 41} and $S(z)$ be a nonzero rational function. Let $\theta\in(0,\pi/2n)$ be a constant. If $f(z)$ is a solution of equation \eqref{Eq 1}, then there is a sequence $\{r_{k}\}$ such that the set $E=\cup_{l=0}^{\infty}[r_{2l},r_{2l+1}]$ has infinite logarithmic measure and for all $r\in E$,
\begin{equation}\label{Eq 3}
\begin{split}
|f(re^{i\theta})|\geq (1-\varepsilon)\frac{\sqrt[n]{\sin n\theta}}{n}r\exp\left(e^{(1-\varepsilon)r^n\cos n\theta}\sin\varepsilon\right).
\end{split}
\end{equation}
\end{theorem}

\vspace*{-10pt}%
\subsection{The Strategy of the Proof}

The proof of Theorem~\ref{maintheorem1}, which will be given in Section~\ref{Proof-1}, follows from an observation on the calculation of Nevanlinna's characteristic $T(r,f)$ of the entire function $e^{e^z}$ in~\cite[pp.~84--85]{Hayman1964}. See also there for the standard notation and basic results of Nevanlinna theory. It is shown there, if $\phi(s)$ is a bounded, non-negative and even function for real $s$ and that
\begin{equation}\label{Eq 3 a5}
\begin{split}
\frac{1}{s}\int_{0}^s\phi(t)dt=l[1+o(1)], \quad s\to\infty,
\end{split}
\end{equation}
then
\begin{equation}\label{Eq 3 a6}
\begin{split}
I(r)=\frac{1}{2\pi}\int_{-\pi}^{\pi}e^{r\cos \theta}\phi(r\sin \theta)d\theta=\frac{[l+o(1)]e^{r}}{(2\pi r)^{1/2}}, \quad r\to\infty.
\end{split}
\end{equation}
If we choose $\phi(s)$ in \eqref{Eq 3 a5} to be $\phi_1(s)=\max\{\cos s,0\}$ or $\phi_2(s)=\max\{-\cos s,0\}$, then the estimate in \eqref{Eq 3 a6} holds for both $\phi_1(s)$ and $\phi_2(s)$. Denote the corresponding integrals by $I_1(r)$ and $I_2(r)$, respectively. Since $e^{e^z}$ has no zeros, by the first main theorem of Nevanlinna we have $\mathbf{I}(r)=I_1(r)-I_2(r)=m(r,e^{e^z})-m(r,e^{-e^z})=\log|f(0)|=1$. Basically, for any real constant $\omega$, by the Cauchy integral formula we have $\oint_{|z|=1}e^{-\omega z}z^{-1}dz=2\pi i$, giving the two identities
\begin{equation}\label{Eq 70-cor4}
\begin{split}
\int_{0}^{2\pi}e^{-\omega\cos\theta}\cos(\omega\sin\theta)d\theta&=2\pi,\\
\int_{0}^{2\pi}e^{-\omega\cos\theta}\sin(\omega\sin\theta)d\theta&=0.
\end{split}
\end{equation}
In general, if the integration interval in \eqref{Eq 3 a6} is not $[-\pi,\pi]$, then the corresponding $\mathbf{I}(r)$ may not be bounded as $r\to\infty$. To provide a lower bound for $|f(z)|$, we will estimate such kind of $\mathbf{I}(r)$ with integration interval being a subinterval of $[0,\pi/2]$ in the proof of Theorem~\ref{maintheorem1}, as well as in the proof of Theorem~\ref{maintheorem2}.

Here we look at the simplest case of equation \eqref{Eq 1}, that is, the case when $h(z)=e^z$. In this case, by integration we may write the solution $f(z)$ in \eqref{Eq 1} as
\begin{equation}\label{Eq 4}
\begin{split}
f(z)=e^{e^z}\left(c+\int_{z_0}^{z}e^{-e^t}dt\right),
\end{split}
\end{equation}
where $c$ is the integration constant and $z_0=x_0+iy_0$ is a fixed constant. Here we choose $z_0=x_0$ to be positive real. We see that $\int_{z_0}^{\infty}e^{-e^t}dt$ converges to a finite nonzero constant, say $d_0$, along the positive real axis. Since $t^2e^{-e^t}$ is decreasing on $[x,\infty)$ for large $x$, we have
\begin{equation}\label{Eq 4nalr-1}
\begin{split}
\int_{x_0}^{x}e^{-e^t}dt=\int_{x_0}^{\infty}e^{-e^t}dt-\int_{x}^{\infty}\left(t^2e^{-e^t}\right)t^{-2}dt=d_0+O\left(xe^{-e^x}\right), \quad x\to\infty.
\end{split}
\end{equation}
It follows that $f(x)=(c+d_0)e^{e^x}+O(x)$ as $x\to\infty$ along the positive real axis. If $c+d_0\not=0$, then $f(z)$ has sufficiently large modulus when $x$ is large. However, the particular case $c+d_0=0$ cannot be excluded out.

Let $\theta\in(0,\pi/2)$ be a constant. Instead of the positive real axis, we consider the ray $\Phi$ on which $z=re^{i\theta}$, $r\in[r_0,\infty)$. For a point $z=x+iy\in \Phi$, we have $y=x\tan\theta$. We connect the two points $z_0=x_0$ and $z_1=x$ by a horizontal line segment $L_0$ and the two points $z_1=x$ and $z=x+iy$ by a vertical line segment $V$ respectively. Let $L$ be the joint of $L_0$ and $V$. For simplicity, denote $\mathbf{r}=e^{x}$. By the Cauchy integral theorem, we may choose the path of integration as $L$ for the solution in \eqref{Eq 4} and it follows that
\begin{equation}\label{Eq 70-cbuy-0}
\begin{split}
f(z)=e^{e^z}\left[c+\int_{x_0}^{x}e^{-e^s}ds+i\mathbf{H}\right],
\end{split}
\end{equation}
where
\begin{equation}\label{Eq 70-cbuy}
\begin{split}
\mathbf{H}=\int_{0}^{y}e^{-\mathbf{r}\cos t}e^{-i\mathbf{r}\sin t}dt=\mathbf{H}_1+i\mathbf{H}_2
\end{split}
\end{equation}
and
\begin{equation}\label{Eq 70-cbuy-1}
\begin{split}
\mathbf{H}_1&=\int_{0}^{y}e^{-\mathbf{r}\cos t}\cos\left(-\mathbf{r}\sin t\right)dt,\\
\mathbf{H}_2&=\int_{0}^{y}e^{-\mathbf{r}\cos t}\sin\left(-\mathbf{r}\sin t\right)dt.
\end{split}
\end{equation}
We may write $y=2\mathbf{k}\pi+\zeta$ with a large integer $\mathbf{k}$ and a number $\zeta\in[-\pi/2,3\pi/2)$. Then we may write $\mathbf{H}_1$ and $\mathbf{H}_2$ as
\begin{equation}\label{Eq 70-cbuy-2}
\begin{split}
\mathbf{H}_1=\mathbf{I}_{11}+\mathbf{I}_{12}, \quad \mathbf{H}_2=\mathbf{I}_{21}+\mathbf{I}_{22},
\end{split}
\end{equation}
where
\begin{equation}\label{Eq 70-cbuy-3}
\begin{split}
\mathbf{I}_{11}&=\int_{0}^{2\mathbf{k}\pi}e^{-\mathbf{r}\cos t}\cos\left(-\mathbf{r}\sin t\right)dt,\\
\mathbf{I}_{12}&=\int_{2\mathbf{k}\pi}^{2\mathbf{k}\pi+\zeta}e^{-\mathbf{r}\cos t}\cos\left(-\mathbf{r}\sin t\right)dt
\end{split}
\end{equation}
and
\begin{equation}\label{Eq 70-cbuy-5}
\begin{split}
\mathbf{I}_{21}&=\int_{0}^{2\mathbf{k}\pi}e^{-\mathbf{r}\cos t}\sin\left(-\mathbf{r}\sin t\right)dt,\\
\mathbf{I}_{22}&=\int_{2\mathbf{k}\pi}^{2\mathbf{k}\pi+\zeta}e^{-\mathbf{r}\cos t}\sin\left(-\mathbf{r}\sin t\right)dt.
\end{split}
\end{equation}
Note that $\int_{x_0}^{x}e^{-e^s}ds=d_0+O(xe^{-e^x})$ as $x\to\infty$ by \eqref{Eq 4nalr-1}. So we need to estimate $\mathbf{I}_{11}$, $\mathbf{I}_{12}$, $\mathbf{I}_{21}$ and $\mathbf{I}_{22}$, respectively. Note that $\cos t$ is even and $\sin t$ is odd and also that $\cos t$ and $\sin t$ are both periodic with period $2\pi$.

Consider first $y=2\mathbf{k}\pi+\zeta$ with a number $\zeta\in[-\pi/2+\varepsilon,\pi/2-\varepsilon]$. For $\mathbf{I}_{11}$, by the first identity in \eqref{Eq 70-cor4}, we have
\begin{equation}\label{Eq yfytq-2}
\begin{split}
\mathbf{I}_{11}=\mathbf{k}\int_{0}^{2\pi}e^{-\mathbf{r}\cos t}\cos\left(\mathbf{r}\sin t\right)dt=2\mathbf{k}\pi.
\end{split}
\end{equation}
For $\mathbf{I}_{12}$, we have
\begin{equation}\label{Eq yfytq-3}
\begin{split}
\left|\mathbf{I}_{12}\right|=\left|\int_{0}^{\zeta}e^{-\mathbf{r}\cos t}\cos\left(\mathbf{r}\sin t\right)dt\right|
\leq\int_{0}^{|\zeta|}\left|e^{-\mathbf{r}\cos t}\cos\left(\mathbf{r}\sin t\right)\right|dt
\leq \left|\zeta e^{-\mathbf{r}\cos \zeta}\right|.
\end{split}
\end{equation}
Obviously, $\mathbf{I}_{12}\to 0$ as $x\to\infty$. Then by \eqref{Eq 70-cbuy-2}, \eqref{Eq yfytq-2} and \eqref{Eq yfytq-3} we see that $|\mathbf{H}_1|\geq (1-\varepsilon)y$. Similarly, we may use the second identity in \eqref{Eq 70-cor4} to show that $\mathbf{I}_{21}=0$ and also that $|\mathbf{I}_{22}|\leq|\zeta e^{-\mathbf{r}\cos \zeta}|$. Thus by \eqref{Eq 70-cbuy-2} we have $\mathbf{H}_2=\mathbf{I}_{22}\to 0$ as $x\to\infty$. Then we see from equations in \eqref{Eq 70-cbuy-0}-\eqref{Eq 70-cbuy-5} that, for each integration constant $c$, the solution $f(z)$ in \eqref{Eq 4} satisfies $|f(z)|\geq (1-\varepsilon)(\sin\theta)re^{\mathbf{r}\cos \zeta}$ for $z=x+iy=re^{i\theta}\in\Phi$ such that $y=2\mathbf{k}\pi+\zeta$ with a large integer $\mathbf{k}$ and a number $\zeta\in[-\pi/2+\varepsilon,\pi/2-\varepsilon]$.

Consider next $y=2\mathbf{k}\pi+\zeta$ with a number $\zeta\in[\pi/2+\varepsilon,3\pi/2-\varepsilon]$. We still have $\mathbf{I}_{11}=2\mathbf{k}\pi$ and $\mathbf{I}_{21}=0$. To estimate $\mathbf{I}_{12}$ and $\mathbf{I}_{22}$, we consider the two cases $\zeta\in[\pi/2+\varepsilon,\pi]$ and $\zeta\in[\pi,3\pi/2-\varepsilon]$, respectively. When $\zeta\in[\pi/2+\varepsilon,\pi]$, by the same arguments as before, we have
\begin{equation}\label{Eq 70-cor4gofyu}
\begin{split}
\left|\mathbf{I}_{12}\right|\leq \int_{2\mathbf{k}\pi}^{2\mathbf{k}\pi+\zeta}\left|e^{-\mathbf{r}\cos t}\cos\left(\mathbf{r}\sin t\right)\right|dt\leq \left|\zeta e^{-\mathbf{r}\cos\zeta}\right|.
\end{split}
\end{equation}
On the other hand, when $\zeta\in[\pi,3\pi/2-\varepsilon]$, we have
\begin{equation*}
\begin{split}
\mathbf{I}_{12}=\int_{2\mathbf{k}\pi}^{2\mathbf{k}\pi+2\pi}e^{-\mathbf{r}\cos t}\cos\left(\mathbf{r}\sin t\right)dt-\int_{2\mathbf{k}\pi+\zeta}^{2\mathbf{k}\pi+2\pi}e^{-\mathbf{r}\cos t}\cos\left(\mathbf{r}\sin t\right)dt,
\end{split}
\end{equation*}
which together with the first identity in \eqref{Eq 70-cor4} yields
\begin{equation}\label{Eq 70-cor4gofyu-2}
\begin{split}
|\mathbf{I}_{12}-2\pi|\leq \left|(2\pi-\zeta)e^{-\mathbf{r}\cos\zeta}\right|.
\end{split}
\end{equation}
Similarly, we may use the second identity in \eqref{Eq 70-cor4} to obtain $|\mathbf{I}_{22}|\leq |\zeta e^{-\mathbf{r}\cos\zeta}|$ or $|\mathbf{I}_{22}|\leq |(2\pi-\zeta)e^{-\mathbf{r}\cos\zeta}|$. Since $|e^{e^z}|=e^{\mathbf{r}\cos\zeta}$ and $\cos\zeta<0$ and since $\mathbf{I}_{11}=2\mathbf{k}\pi=y-\zeta$ and $\mathbf{I}_{21}=0$, by the inequalities in \eqref{Eq 70-cor4gofyu} and \eqref{Eq 70-cor4gofyu-2} together with similar inequalities for $|\mathbf{I}_{22}|$, we see from equations in \eqref{Eq 70-cbuy-0}-\eqref{Eq 70-cbuy-5} that, for each integration constant $c$, the solution $f(z)$ in \eqref{Eq 4} satisfies $|f(z)|\leq 2\pi$ for $z=x+iy\in\Phi$ such that $y=2\mathbf{k}\pi+\zeta$ with a large integer $\mathbf{k}$ and a number $\zeta\in[\pi/2+\varepsilon,3\pi/2-\varepsilon]$.

On the basis of the previous estimates for $|f(z)|$ of $f(z)$ in \eqref{Eq 4} along the ray $\Phi$, we define a sequence $\{r_k\}$ in the way that $r_{2l}\sin\theta=2(2l+\mathbf{k})\pi-\pi/2+\varepsilon$ if $k=2l$ and $r_{2l+1}\sin\theta=2(2l+\mathbf{k})\pi+\pi/2-\varepsilon$. Then the inequality in \eqref{Eq 3} holds for all $z=x+iy=re^{i\theta}$ such that $r \in E=\cup_{l=0}^{\infty}[r_{2l},r_{2l+1}]$. Since $(r_{2l+1}-r_{2l})\sin\theta=\pi-2\varepsilon$, the set $E=\cup_{l=0}^{\infty}[r_{2l},r_{2l+1}]$ has infinite logarithmic measure, i.e., $\int_{E}dr/r=\infty$.

\subsection{Plan of the Paper}

The remainder of this paper is structured as follows. In Section~\ref{Proof-1}, we prove Theorem~\ref{maintheorem1} in the general case of $h(z)$. For a constant $\theta\in(0,\pi/2n)$, we will consider the ray $\Phi$ on which $z=re^{i\theta}$, $r\in[r_0,\infty)$. We would like to define the corresponding $\mathbf{H}$ as in the simplest case of $h(z)$ and estimate the modulus of $\mathbf{H}$. To this end, in Subsection~\ref{Part I} we first consider the asymptotic behavior of the function
\begin{equation}\label{Eq 63hil}
\begin{split}
U(z)=\int_{z_0}^zS(t)e^{P(t)}dt,
\end{split}
\end{equation}
where $z_0$ is fixed and $|z_0|=r_0>\dot{r}$, and write $U(z)=e^{\mathbf{u}(x,y)+i\mathbf{v}(x,y)}$ with two functions $\mathbf{u}(x,y)$ and $\mathbf{v}(x,y)$ in a carefully chosen region of the plane. In so doing, we are able to write $\mathbf{H}=\int_{\mathbf{y}_0}^{y}e^{-e^{\mathbf{u}(x,t)+i\mathbf{v}(x,t)}}dt$ for a point $\mathbf{y}_0$ such that $\mathbf{v}(x,\mathbf{y}_0)=0$. We note that $\mathbf{H}$ is too sensitive to the perturbation of $t$ when $x$ is large and $\mathbf{v}(x,t)<0$. In Subsection~\ref{Part II} we shall choose a simple smooth curve $\Omega$ which is asymptotically close to the ray $z=re^{i\pi/2n}$, $r\in[r_0,\infty)$ and on which $e^{\mathbf{u}(x,y)}=\omega$ for a constant $\omega>0$ and also choose a family of simple smooth curves $L_k$, $k=0,1,\cdots$ on which $\mathbf{v}(x,y)=2k\pi$. Note that each $L_k$ is perpendicular to $\Omega$. By integrating $e^{-U(z)}$ along the curve $\Omega$ and along each curve $L_k$ respectively, we use the Cauchy integral theorem to relate the value of $\mathbf{H}$ to the values of $\int_{\Omega}e^{-U(t)}dt$ and $\int_{L_k}e^{-U(t)}dt$. In~Subsection~\ref{Part III}, we estimate the value of $\int_{L_k}e^{-U(t)}dt$ and also, using the two identities in \eqref{Eq 70-cor4}, estimate the value of $\int_{\Omega}e^{-U(t)}dt$ respectively. In~Subsection~\ref{Part IV}, we obtain estimates for $|f(z)|$ for $z$ in most parts of $\Phi$ and in particular define the desired set $E$. In Section~\ref{Lower bound} we will further look at the modulus $|f(z)|$ of $f(z)$ in \eqref{Eq 4} and provide a lower bound for $|f(z)|$ for $z=x+iy\in\Phi$ such that $y=2\mathbf{k}\pi+\zeta$ with a large integer $\mathbf{k}$ and a number $\zeta\in[\pi/2+\varepsilon,3\pi/2-\varepsilon]$. Finally, in Section~\ref{Concluding remarks} we illustrate the applications of the results of Theorem~\ref{maintheorem1} and Theorem~\ref{maintheorem2} and, meanwhile, extend the method in the proof of Theorem~\ref{maintheorem1} to more general case.

\section{Proof of Theorem~\ref{maintheorem1}}\label{Proof-1}

Let $f(z)$ be a solution of equation \eqref{Eq 1}. Then, outside the disc $D(0,\dot{r})$, by integration we may write $f(z)$ in the form
\begin{equation}\label{Eq 63}
\begin{split}
f(z)=e^{U(z)}\left[c+\int_{z_0}^{z}e^{-U(t)}dt\right],
\end{split}
\end{equation}
where $U(z)$ is defined in \eqref{Eq 63hil}, $c$ is the integration constant and $z_0=x_0+iy_0$. In the following, we always suppose that $x_0$ is large enough. Our proof will be divided into four parts.

\subsection{Preliminaries}\label{Part I}

To give a more specialised form for $U(z)$ in \eqref{Eq 63hil}, in this subsection we first consider the asymptotic behavior of the function
\begin{equation}\label{Eq 43}
\begin{split}
G(z)=e^{-P(z)}\int_{z_0}^zS(t)e^{P(t)}dt.
\end{split}
\end{equation}
Denote $\theta_k=\frac{2k-1}{2n}\pi$, $k=0,1,\cdots,2n$ and let
\begin{equation}\label{Eq 43ab}
\begin{split}
A_k=\left\{z=re^{i\theta}: \quad r\in[r_0,\infty), \quad \theta\in(\theta_{k},\theta_{k+1})\right\}.
\end{split}
\end{equation}
Note that $\theta_{2n}-\theta_{0}=2\pi$. Then along the ray $z=re^{i\theta}$, $r\in[r_0,\infty)$, if $\theta\in(\theta_{k},\theta_{k+1})$ and $\cos n\theta>0$, then $\log |e^{P(z)}|$ is increasing on $[r_0,\infty)$ and $|e^{P(z)}|\geq e^{r^{n}/2}$ there; if $\theta\in(\theta_{k},\theta_{k+1})$ and $\cos n\theta<0$, then $\log |e^{P(z)}|$ is decreasing on $[r_0,\infty)$ and $|e^{P(z)}|\leq e^{-r^{n}/2}$ there. See~\cite{Banklangley1987} and also \cite[Lemma~5.14]{Laine1993}. Moreover, as shown by Bank and Langley \cite{Banklangley1987} (see also \cite{zhang2021,zhang2021-1,Zhang2022}), we have the following

\begin{lemma}\label{Lemma 0}
For the function $G(z)$ defined in \eqref{Eq 43}, there is a rational function $Q(z)$ and a large integer $N$ such that along the ray $z=re^{i\theta}$, $r\in[r_0,\infty)$ and $\theta\in(\theta_{k},\theta_{k+1})$ and $\cos n\theta<0$,
\begin{equation}\label{Eq 44}
\begin{split}
G(re^{i\theta})=Q(re^{i\theta})+c_1e^{-P(z)}+O\left(r^{-N}\right),
\quad r\to\infty,
\end{split}
\end{equation}
where $c_1=c_1(\theta)$ is a constant, while along the ray $z=re^{i\theta}$, $r\in[r_0,\infty)$ and $\theta\in(\theta_{k},\theta_{k+1})$ and $\cos n\theta>0$,
\begin{equation}\label{Eq 45}
\begin{split}
G(re^{i\theta})=Q(re^{i\theta})+O\left(r^{-N}\right), \quad r\to\infty.
\end{split}
\end{equation}
\end{lemma}

In the original proof of Lemma~\ref{Lemma 0} in~\cite{Banklangley1987}, the error terms in the two estimates \eqref{Eq 44} and \eqref{Eq 45} take the form $O(r^{-2})$ or $O(r^{-1})$, but their proof gives the error term $O(r^{-N})$. Moreover, by the Phragm\'{e}n--Lindel\"{o}f theorem (see \cite[theorem~7.3]{hollandasb}), the constant $c_1$ in \eqref{Eq 44} remains the same for any $\theta$ such that $\theta\in(\theta_{k},\theta_{k+1})$ and $\cos n\theta<0$.

To prove Theorem~\ref{maintheorem1}, we also need to consider the asymptotic behavior of $G(z)$ along curves that are asymptotically close to the ray $z=re^{i\pi/2n}$. Here and in the following, we always assume that curves appearing are simple and smooth. Let $\omega>0$ be a constant. We shall define a slightly larger domain than $A_{0}$ in \eqref{Eq 43ab} as
\begin{equation}\label{Eq 43abarefq}
\begin{split}
A_{\omega}=\left\{z=re^{i\theta}: \left|\theta\right|\leq \frac{\pi}{2n}+\varepsilon, \quad r\in[r_0,\infty), \quad \left|e^{P(z)}\right|>\frac{\omega}{2|z|^{2|m-n+1|}}\right\}.
\end{split}
\end{equation}
We may of course suppose that $z_0\in A_{\omega}$. To deal with $|e^{P(z)}|$, it will be convenient to write the polynomial $P(z)$ in \eqref{Eq 41} as
\begin{equation}\label{Eq 65}
\begin{split}
P(z)=u(x,y)+iv(x,y).
\end{split}
\end{equation}
Note that $v(x,y)$ is the \emph{harmonic conjugate} of $u(x,y)$. Denote by $\epsilon(x)$ any real quantity such that $|\epsilon(x)|=O(x^{-1})$ as $x\to\infty$ but does not necessarily the same in each appearance. The notation $\epsilon(y)$ is understood in the same way. Then we prove the following

\begin{lemma}\label{Lemma 1}
For the function $G(z)$ defined in \eqref{Eq 43}, there is a rational function $Q(z)$ and a large integer $N>2|m-n+1|$ such that
\begin{equation}\label{Eq 45gkyu}
\begin{split}
G(re^{i\theta})=Q(re^{i\theta})+\left[c_2-Q(z_0)e^{P(z_0)}\right]e^{-P(z)}+O\left(r^{-N}\right)
\end{split}
\end{equation}
uniformly as $z=re^{i\theta}\to\infty$ in $A_{\omega}$, where $c_2=c_2(\theta)$ is a constant dependent on $\theta$.
\end{lemma}

\begin{proof}
We define two sequence $\{S_j\}_{j\in \mathbb{N}}$ and $\{Q_j\}_{j\in \mathbb{N}}$ inductively by
\begin{equation*}
\begin{split}
S_1=\frac{S}{P'}, \quad Q_1=-S_1', \quad \cdots, \quad S_{j+1}=\frac{Q_j}{P'}, \quad Q_{j+1}=(-1)^jS'_{j+1}, \quad j\geq 1.
\end{split}
\end{equation*}
Recall that $S(z)=nz^{m}[1+O(z^{-1})]$ as $z\to\infty$ for some integer $m$. Then, by looking at the degree $n$ of $P(z)$ and the integer $m$, we see immediately that
\begin{equation}\label{Eq 47}
\begin{split}
\left|Q_k(z)\right|\leq B_k|z|^{m-kn}
\end{split}
\end{equation}
holds for a constant $B_k>0$ provided that $|z|$ is sufficiently large, say $|z|\geq r_k$. Take now the integral from the right-hand side of equation \eqref{Eq 43}. Integration by parts results, for each $k\in \mathbb{N}$, in
\begin{equation*}
\begin{split}
\int_{z_0}^zS(t)e^{P(t)}dt&=\left[S_1(t)e^{P(t)}\right]_{z_0}^{z}+\int_{z_0}^zQ_1(t)e^{P(t)}dt=\cdots\\
&=\left[(S_1(t)+\cdots+S_k(t))e^{P(t)}\right]_{z_0}^{z}+\int_{z_0}^zQ_k(t)e^{P(t)}dt.
\end{split}
\end{equation*}
If $Q_k(z)\equiv0$, we have just to define $Q=S_1(z)+\cdots+S_k(z)$. Otherwise, we select $k$ large enough so that $m-kn\leq -(N+1)$ for an integer $N$. Then we have to estimate $\int_{z_0}^zQ_k(t)e^{P(t)}dt$.

Instead of integrating $Q_k(t)e^{P(t)}$ along a ray from the point $z_0$ to $z$, we might as well take a new path of integration as in \cite[Lemma~4.1]{Jan-Martin2005}, namely, we first integrate $Q_k(t)e^{P(t)}$ from the point $z_0$ to $\infty$ along the ray $\Omega_1$ on which $t=z_0+re^{i\pi/n}$ and then come back along the ray $\Omega_2$ on which $t=w+re^{i\pi/n}$ from $\infty$ to a point $w$ such that $|z|=2|w|$ and finally move along a curve $\Omega_3$ from $w$ to $z$. By the Cauchy--Riemann equations, we have
\begin{equation*}
\begin{split}
P'(z)=u_x-iu_y=nz^{n-1}\left[1+O\left(z^{-1}\right)\right],
\end{split}
\end{equation*}
where $u_x$ and $u_y$ are the partial derivatives of $u=u(x,y)$ with respect to $x$ and $y$ respectively. By comparing the real and imaginary parts on both sides of the above equation, we find
\begin{equation*}
\begin{split}
u_x&=[1+\epsilon(x)]\text{Re}(nz^{n-1})+\epsilon(x)\text{Im}(nz^{n-1}),\\
u_y&=\epsilon(x)\text{Re}(nz^{n-1})-[1+\epsilon(x)]\text{Im}(nz^{n-1}).
\end{split}
\end{equation*}
For the point $z=re^{i\theta}\in\Omega_3$, we may write $nz^{n-1}=nr^{n-1}e^{i(n-1)\theta}=nr^{n-1}\cos(n-1)\theta+inr^{n-1}\sin(n-1)\theta$. In particular, when $n=1$, by the expression of $P(z)$ in \eqref{Eq 41} we have $u(x,y)=x$ and $v(x,y)=y$ and thus $u_x=1$ and $u_y=0$; when $n\geq 2$, we see that $u_x>0$ and $u_y<0$ when $|z|$ is large. Thus, on the curve $\Omega_3$ we may always suppose that $e^{P(z)}=e^{u(x,y)}e^{iv(x,y)}$ has the maximum modulus at the point $z$. Then, by the previous observation on the asymptotic behavior of $|e^{P(z)}|$, we see that
\begin{equation*}
\begin{split}
\int_{\Omega_1\cup\Omega_2}Q_k(t)e^{P(t)}dt=c_2(\theta)
\end{split}
\end{equation*}
converges. Note that $c_2(\theta)$ tends to a fixed constant as $z\to\infty$ in $A_{\omega}$. Thus we have
\begin{equation}\label{Eq 50}
\begin{split}
\int_{z_0}^{z}Q_k(t)e^{P(t)}dt=c_2(\theta)-\int_{\Omega_3}Q_k(t)e^{P(t)}dt.
\end{split}
\end{equation}
By \eqref{Eq 47} and the choice of $k$ we see that
\begin{equation}\label{Eq 51}
\begin{split}
\left|Q_k(z)e^{P(z)}\right|\leq B_k\left|e^{P(z)}\right||z|^{-(N+1)}
\end{split}
\end{equation}
for all $|z|$ sufficiently large. Define again $Q=S_1(z)+\cdots+S_k(z)$. By the choice of $\Omega_3$, we see that the length of the curve $\Omega_3$ is always of type $O(r)$, uniformly for all $z\in A_{\omega}$. Thus, by \eqref{Eq 51}, we have
\begin{equation}\label{Eq 52}
\begin{split}
\left|\int_{\Omega_3}Q_k(t)e^{P(t)}dt\right|\leq B_k\left|e^{P(z)}\right|O\left(|z|^{-N}\right)
\leq C_k\left|e^{P(z)}\right|r^{-N}
\end{split}
\end{equation}
for all large $r$ and some constant $C_{k}$. Thus we have \eqref{Eq 45gkyu} by \eqref{Eq 50} and \eqref{Eq 52} together with the definition of $G(z)$ in \eqref{Eq 43}. This completes the proof.
\end{proof}

In the following we will consider the solution in \eqref{Eq 63} along a ray $\Phi$ on which $z=re^{i\theta}$, $r\in[r_0,\infty)$ and $\theta\in(0,\pi/2n)$ or along some curve $\Omega$ which is asymptotically close to the ray $z=re^{i\pi/2n}$. Both $\Phi$ or $\Omega$ are contained in $A_{\omega}$. By Lemma~\ref{Lemma 1}, the function $G(z)$ defined in \eqref{Eq 43} can be written, for $z=re^{i\theta}$ in either $\Phi$ or $\Omega$, as
\begin{equation*}
\begin{split}
G(z)=Q(z)+c_3e^{-P(z)}+H(z),
\end{split}
\end{equation*}
where $c_3=c_2-Q(z_0)e^{P(z_0)}$ is a constant and
\begin{equation*}
\begin{split}
H(z)=e^{-P(z)}\int_{z_0}^zQ_k(t)e^{P(t)}dt
\end{split}
\end{equation*}
satisfies $H(z)=O(r^{-N})$ as $z\to\infty$ along either $\Phi$ or $\Omega$. By the definition $U(z)=e^{P(z)}G(z)$, we may rewrite the integration constant $c$ in \eqref{Eq 63} as $ce^{-c_3}$. For any $\theta$, this redefinition does not affect the estimation of $|f(z)|$ in the following. Now, by our assumptions on $P(z)$ and $S(z)$, we see from Lemma~\ref{Lemma 1} that
\begin{equation}\label{Eq 47jfaop-9yhi}
\begin{split}
G(z)=z^{m-n+1}\left[1+\frac{Q(z)-z^{m-n+1}}{z^{m-n+1}}+\frac{H(z)}{z^{m-n+1}}\right]=z^{m-n+1}\left[1+O\left(z^{-1}\right)\right]
\end{split}
\end{equation}
along the aforementioned $\Phi$ or $\Omega$. By redefining the integration constant $c$ in \eqref{Eq 63}, we may write $U(z)$ in \eqref{Eq 63hil} uniformly as
\begin{equation*}
\begin{split}
U(z)=e^{P(z)}G(z)=e^{P(z)+R(z)}
\end{split}
\end{equation*}
where $R(z)$ is chosen as an analytic branch of the function $\ln G(z)=\ln(Q(z)+H(z))$. Similar as for the expression of $P(z)$ in \eqref{Eq 65}, by \eqref{Eq 47jfaop-9yhi} we may also write $R(z)$ as
\begin{equation}\label{Eq 66}
\begin{split}
R(z)&=p(x,y)+iq(x,y)\\
&=(m-n+1)\ln\sqrt{x^2+y^2}+\epsilon(y)+i\left[(m-n+1)\arctan\frac{y}{x}+\epsilon(y)\right].
\end{split}
\end{equation}
Also, $q(x,y)$ is the harmonic conjugate of $p(x,y)$. Denote $\mathbf{u}(x,y)=u(x,y)+p(x,y)$ and $\mathbf{v}(x,y)=v(x,y)+q(x,y)$ for simplicity. By the definition of $A_{\omega}$ in \eqref{Eq 43abarefq}, we may define $U(z)=e^{\mathbf{u}(x,y)+i\mathbf{v}(x,y)}$ around a point $z=x+iy\in \Phi$ and then extend the definition to the whole $A_{\omega}$ by analytic continuation. Thus, in $A_{\omega}$, we may always write $U(z)$ as
\begin{equation}\label{Eq 63a}
\begin{split}
U(z)=e^{u(x,y)+p(x,y)+i[v(x,y)+q(x,y)]}=e^{\mathbf{u}(x,y)}\cos(\mathbf{v}(x,y))+ie^{\mathbf{u}(x,y)}\sin(\mathbf{v}(x,y)).
\end{split}
\end{equation}
Below we shall also use the suppressed notation $\mathbf{u}=\mathbf{u}(x,y)$ and $\mathbf{v}=\mathbf{v}(x,y)$ for simplicity.

Suppose that $\mathbf{u}$ and $\mathbf{v}$ have been defined. We look at $\mathbf{u}$ and $\mathbf{v}$ more carefully. By \eqref{Eq 65} and \eqref{Eq 66}, along the ray $\Phi$ we have
\begin{equation}\label{Eq 66vfiytdyst}
\begin{split}
\mathbf{u}=\frac{\cos n\theta}{(\sin\theta)^n}y^n\left[1+\epsilon(y)+O(y^{-n}\ln y)\right]
\end{split}
\end{equation}
and, along the ray $\Phi$ and the curve $\Omega$, we have
\begin{equation}\label{Eq 66vfiytdystbk}
\begin{split}
\mathbf{v}=\frac{\sin n\theta}{(\sin\theta)^n}y^n\left[1+\epsilon(y)\right]
\end{split}
\end{equation}
Note that $\theta$ is fixed as $z\to\infty$ along the curve $\Phi$ while $\theta$ is not necessarily fixed as $z\to\infty$ along the curve $\Omega$. Also note that $\theta\to \pi/2n$ as $z\to\infty$ along the curve $\Omega$.

Denote by $\mathbf{u}_y$ and $\mathbf{v}_y$ the partial derivatives of $\mathbf{u}$ and $\mathbf{v}$ with respect to $y$ respectively. We give asymptotic expressions for $\mathbf{v}_y$ along the ray $\Phi$ and the curve $\Omega$. By the two expressions in \eqref{Eq 63hil} and \eqref{Eq 63a} for $U(z)$, it is elementary to take the first derivative of $U(z)$ and obtain
\begin{equation}\label{Eq 63b}
\begin{split}
S(z)e^{P(z)}=e^{\mathbf{u}}\left[\mathbf{v}_y\cos\mathbf{v}+\mathbf{u}_y\sin\mathbf{v}+i(\mathbf{v}_y\sin\mathbf{v}-\mathbf{u}_y\cos\mathbf{v})\right].
\end{split}
\end{equation}
Moreover, by \eqref{Eq 47jfaop-9yhi} we have
\begin{equation}\label{Eq 63c}
\begin{split}
S(z)e^{P(z)}=\frac{S(z)}{G(z)}\left[e^{P(z)}G(z)\right]=\left[1+O\left(\frac{1}{z}\right)\right]nz^{n-1}e^{\mathbf{u}+i\mathbf{v}}, \quad z\to\infty.
\end{split}
\end{equation}
Note that $1+O(z^{-1})=1+\epsilon(y)+i\epsilon(y)$. We compare the real and imaginary parts on both sides of the two equations in \eqref{Eq 63b} and \eqref{Eq 63c} and find
\begin{equation}\label{Eq 77 a1 fu1aojr-3hilu}
\begin{split}
\mathbf{v}_y\cos\mathbf{v}+\mathbf{u}_y\sin\mathbf{v}=\text{Re}(nz^{n-1})&\left\{\left[1+\epsilon(y)\right]\cos\mathbf{v}+\epsilon(y)\sin\mathbf{v}\right\}\\
&-\text{Im}(nz^{n-1})\left\{\left[1+\epsilon(y)\right]\sin\mathbf{v}+\epsilon(y)\cos\mathbf{v}\right\}
\end{split}
\end{equation}
and
\begin{equation}\label{Eq 77 a1 fu1aojr-3hilu-1}
\begin{split}
\mathbf{v}_y\sin\mathbf{v}-\mathbf{u}_y\cos\mathbf{v}=\text{Re}(nz^{n-1})&\left\{\left[1+\epsilon(y)\right]\sin\mathbf{v}+\epsilon(y)\cos\mathbf{v}\right\}\\
&+\text{Im}(nz^{n-1})\left\{\left[1+\epsilon(y)\right]\cos\mathbf{v}+\epsilon(y)\sin\mathbf{v}\right\}.
\end{split}
\end{equation}
We multiply by $\sin\mathbf{v}$ on both sides of \eqref{Eq 77 a1 fu1aojr-3hilu} and by $\cos\mathbf{v}$ on both sides of \eqref{Eq 77 a1 fu1aojr-3hilu-1} and then make subtraction on both sides of the resulting two equations to obtain
\begin{equation}\label{Eq 77 a1 fu1aojr-3hilu-2giu-1}
\begin{split}
\mathbf{u}_y=\epsilon(y)\text{Re}(nz^{n-1})-[1+\epsilon(y)]\text{Im}(nz^{n-1}).
\end{split}
\end{equation}
Similarly, we multiply by $\cos\mathbf{v}$ on both sides of \eqref{Eq 77 a1 fu1aojr-3hilu} and by $\sin\mathbf{v}$ on both sides of \eqref{Eq 77 a1 fu1aojr-3hilu-1} and then make addition on both sides of the resulting two equations to obtain
\begin{equation}\label{Eq 77 a1 fu1aojr-3hilu-2giu-2}
\begin{split}
\mathbf{v}_y=[1+\epsilon(y)]\text{Re}(nz^{n-1})+\epsilon(y)\text{Im}(nz^{n-1}).
\end{split}
\end{equation}
Along the curve $\Phi$ and the curve $\Omega$, we may give explicit expressions for $\text{Re}(nz^{n-1})$ and $\text{Im}(nz^{n-1})$ as
\begin{equation}\label{Eq 66vfiytdyst-1}
\begin{split}
\text{Re}(nz^{n-1})&=n\frac{\cos (n-1)\theta}{(\sin\theta)^{n-1}}y^{n-1},\\
\text{Im}(nz^{n-1})&=n\frac{\sin (n-1)\theta}{(\sin\theta)^{n-1}}y^{n-1}.
\end{split}
\end{equation}
Then from \eqref{Eq 77 a1 fu1aojr-3hilu-2giu-1}, \eqref{Eq 77 a1 fu1aojr-3hilu-2giu-2} and \eqref{Eq 66vfiytdyst-1} we see that along the ray $\Phi$, as well as along the curve $\Omega$, we always have
\begin{equation}\label{Eq 77 a1 fu1aojr-3hilu-1-fu}
\begin{split}
\mathbf{v}_y=\text{Re}(nz^{n-1})[1+\epsilon(y)]=\frac{n\cos (n-1)\theta}{(\sin\theta)^{n-1}}y^{n-1}[1+\epsilon(y)]
\end{split}
\end{equation}
and also that
\begin{equation}\label{Eq 77 a1 fu1aojr-3hilu-2}
\begin{split}
\frac{\mathbf{u}_y}{\mathbf{v}_y}=\epsilon(y)
\end{split}
\end{equation}
when $n=1$ and
\begin{equation}\label{Eq 77 a1 fu1aojr-3hilu-3}
\begin{split}
\frac{\mathbf{u}_y}{\mathbf{v}_y}=-\frac{\sin(n-1)\theta}{\cos(n-1)\theta}[1+\epsilon(y)]
\end{split}
\end{equation}
when $n\geq 2$. Here we note that, along the ray $\Phi$, in the estimates in \eqref{Eq 77 a1 fu1aojr-3hilu}-\eqref{Eq 77 a1 fu1aojr-3hilu-3}, we may replace $y$ by $y=x\tan\theta$ and, in particular, replace the quantities $\epsilon(y)$ there by $\epsilon(x)$. Along the curve $\Omega$, we may also do this when $n\geq 2$, but this is not allowed when $n=1$.

For clarity, in the following subsections we shall distinguish the points in $\Phi$ and $\Omega$ by particularly writing $z=re^{i\vartheta}$ for the point $z$ in the curve $\Omega$ which will be defined later. Moreover, we shall use the notations
\begin{equation}\label{Eq 77 notation}
\begin{split}
\iota&=\frac{\sin(n-1)\theta}{\cos(n-1)\theta}, \qquad \iota_1=\frac{n\cos (n-1)\theta}{(\sin\theta)^{n-1}},\\
\kappa&=\frac{\sin(n-1)\vartheta}{\cos(n-1)\vartheta}, \qquad \kappa_1=\frac{n\cos (n-1)\vartheta}{(\sin\vartheta)^{n-1}}
\end{split}
\end{equation}
in both of the two cases $n=1$ and $n\geq 2$. Note that we always have $\iota>0$ and $\kappa>0$ when $n\geq 2$.

\subsection{Definition of the path of integration}\label{Part II}

From now on, we fix a constant $\theta\in(0,\pi/2n)$. Then we define the path of integration for $f(z)$ in \eqref{Eq 63} as the ray $\Phi$ on which
\begin{equation*}
\begin{split}
z=x+iy, \quad y=x\tan\theta, \quad x\in[x_0,\infty).
\end{split}
\end{equation*}
In addition to $\Phi$, we also define the curves $\Psi_k$ and $\Omega$ as follows. By the definition of $\mathbf{v}$, for each integer $k\geq 0$, there is a curve $\Psi_k$ on which
\begin{equation*}
\begin{split}
z=x+iy, \quad \mathbf{v}(x,y)=2k\pi, \quad x\in[x_k,\infty).
\end{split}
\end{equation*}
We may of course suppose that $z_0\in \Psi_0$. Also, for the constant $\omega>0$ in $A_{\omega}$, there is a curve $\Omega$ on which
\begin{equation*}
\begin{split}
z=x+iy, \quad e^{\mathbf{u}(x,y)}=\omega, \quad x\in[x_0,\infty).
\end{split}
\end{equation*}
By substituting $z=re^{i\vartheta}$, $0<\vartheta<\pi/n$ into the equation $|e^{P(z)+R(z)}|=\omega$, we find $\cos n\vartheta=\epsilon(y)+O(y^{-n}\ln y)$, implying that $\vartheta=[1+\epsilon(y)+O(y^{-n}\ln y)]\pi/2n$ as $z\to\infty$ along the curve $\Omega$. Recalling the definition of $A_{\omega}$ in \eqref{Eq 43abarefq}, we may suppose that $\Phi$, $\Psi_k$ and $\Omega$ all lie entirely in $A_{\omega}$. It is elementary that each $\Psi_k$ is perpendicular to $\Omega$. By analytic continuation, we may suppose that the curve $\Psi_k$ with the curve $\Omega$ at the unique point $z_k=x_k+iy_k$, $k=0,1,\cdots$, for the first time.

For clarity, for a fixed point $z=x+iy\in \Phi$, below we shall write $\tilde{\mathbf{u}}=\mathbf{u}(s,t)$ and $\tilde{\mathbf{v}}=\mathbf{v}(s,t)$ which differ from $\mathbf{u}=\mathbf{u}(x,y)$ and $\mathbf{v}=\mathbf{v}(x,y)$ by the variables $s$ and $t$. We connect the point $\mathbf{z}_0=x+i\mathbf{y}_{0}\in L_0$ and the point $z=x+iy\in\Phi$ by a vertical line segment $V$. By the definition of $\mathbf{v}$ together with the expressions in \eqref{Eq 65} and \eqref{Eq 66}, we easily see that $\mathbf{y}_0=\epsilon(x)$ as $z\to\infty$ along the curve $\Psi_0$. We claim that $V$ intersects with each $\Psi_k$ at some unique point $\mathbf{z}_k=x+i\mathbf{y}_k$. In fact, when $n=1$, this is obvious since $\tilde{\mathbf{v}}=t[1+\epsilon(t)]$ as $t\to\infty$ and thus we may choose $\Psi_k$ in the way that $t=2k\pi[1+\epsilon(t)]$. When $n\geq 2$, recall the relation for $\mathbf{v}$ along $\Phi$ and $\Omega$ in \eqref{Eq 66vfiytdystbk}. By taking the first derivative on both sides of the equation $\tilde{\mathbf{v}}=\mathbf{v}(s,t)=2k\pi$ with respect to $s$ along the curve $\Psi_k$, we get
\begin{equation}\label{Eq 6aren;f}
\begin{split}
\tilde{\mathbf{v}}_s+\tilde{\mathbf{v}}_t\frac{dt}{ds}=0.
\end{split}
\end{equation}
Since $\mathbf{u}$ and $\mathbf{v}$ are the real and imaginary parts of an analytic function, we have $\mathbf{u}_y=-\mathbf{v}_x$ by the Cauchy--Riemann equations. This yields
\begin{equation}\label{Eq 6aren;f-1}
\begin{split}
\frac{dt}{ds}=\frac{\tilde{\mathbf{u}}_t}{\tilde{\mathbf{v}}_t}.
\end{split}
\end{equation}
In particular, along the part of the curve $\Psi_k$ between $\Omega$ and $\Phi$, denoted by $\Psi_k^{l}$, by the estimate \eqref{Eq 77 a1 fu1aojr-3hilu-3} for $\mathbf{u}_y/\mathbf{v}_y$ we see that $dt/ds<0$ for large $t$; along the part of the curve $\Psi_k$ in the right of $\Phi$, denoted by $\Psi_k^{r}$, we may replace the quantity $\epsilon(y)$ in \eqref{Eq 77 a1 fu1aojr-3hilu-3} by $\epsilon(x)$ and also see that $dt/ds<0$ for large $s$. We conclude that $t$ decreases as $t\to\infty$ along each curve $\Psi_k$. Thus the ray $\Phi$ intersects with each curve $\Psi_k$ at some unique point and also that $V$ intersects with each curve $\Psi_k$ at some unique point $\mathbf{z}_k=x+i\mathbf{y}_k$ for $k=0,1,2,\cdots,\mathbf{k}$ and some integer $\mathbf{k}$ dependent on $x$ such that $0\leq \mathbf{v}(x,y)-\mathbf{v}(x,\mathbf{y}_{\mathbf{k}})<2\pi$.

Here we note that, for the integer $\mathbf{k}$, if $\tilde{z}=x+i\tilde{y}\in A_{\omega}$ is a point between $\Psi_{\mathbf{k}-1}$ and $\Psi_{\mathbf{k}}$ or between $\Psi_{\mathbf{k}}$ and $\Psi_{\mathbf{k}+1}$, then
\begin{equation}\label{Eq 77 a1 fu1-a-ijg}
\begin{split}
\mathbf{u}(x,y)-\mathbf{u}(x,\tilde{y})=\frac{\mathbf{u}_y}{\mathbf{v}_y} \left[\mathbf{v}(x,y)-\mathbf{v}(x,\tilde{y}\right]\left[1+\epsilon(x)\right].
\end{split}
\end{equation}
In fact, denoting $\Delta y=y-\tilde{y}$, by Lagrange's mean value theorem, there are two constants $\delta_1,\delta_2\in(0,1)$ such that
\begin{equation}\label{Eq 77 a1 funk}
\begin{split}
\mathbf{u}(x,y)-\mathbf{u}(x,\tilde{y})&=\mathbf{u}_y(x,y+\delta_1\Delta y)\Delta y,\\
\mathbf{v}(x,y)-\mathbf{v}(x,\tilde{y})&=\mathbf{v}_y(x,y+\delta_2\Delta y)\Delta y.
\end{split}
\end{equation}
By the two estimates in \eqref{Eq 77 a1 fu1aojr-3hilu-2} and \eqref{Eq 77 a1 fu1aojr-3hilu-3} for $\mathbf{u}_y/\mathbf{v}_y$, we may write
\begin{equation}\label{Eq 77 a1 funk-1}
\begin{split}
\frac{\mathbf{u}_y(x,y+\delta_1\Delta y)}{\mathbf{v}_y(x,y+\delta_2\Delta y)}=\frac{\mathbf{u}_y(x,y)}{\mathbf{v}_y(x,y)}[1+\epsilon(x)]
\end{split}
\end{equation}
in both two cases $n=1$ and $n\geq 2$. Then the relation in \eqref{Eq 77 a1 fu1-a-ijg} follows from \eqref{Eq 77 a1 funk} and \eqref{Eq 77 a1 funk-1}. In particular, when $n=1$, by \eqref{Eq 77 a1 fu1aojr-3hilu-2} we have $\mathbf{u}(x,y)-\mathbf{u}(x,\tilde{y})=\epsilon(x)$.

For each $k=0,1,\cdots,\mathbf{k}$, we denote by $L_k$ the part of $\Psi_k$ between the point $z_k$ and the point $\mathbf{z}_k$ and by $\Omega_{k}$ the part of $\Omega$ between the two points $z_{0}$ and $z_{k}$, respectively. We integrate $e^{-U(z)}$ along the curve $L_k$ and the curve $\Omega_k$ respectively and denote
\begin{equation}\label{Eq ye4evrfnwli-1}
\begin{split}
\mathbf{F}_k=\int_{L_k}e^{-U(z)}dz
\end{split}
\end{equation}
and
\begin{equation}\label{Eq ye4evrfnwli-1-1}
\begin{split}
\mathbf{G}_k=\int_{\Omega_{k}}e^{-U(z)}dz.
\end{split}
\end{equation}
In particular, let $L$ be the joint of $L_0$ and $V$. Then, along the curve $L$, by the Cauchy integral theorem, we may write the solution in \eqref{Eq 63} as
\begin{equation}\label{Eq 76fiyt}
\begin{split}
f(z)=e^{U(z)}\left(c+\mathbf{F}_0+i\mathbf{H}\right),
\end{split}
\end{equation}
where
\begin{equation}\label{Eq 76fiyt-1}
\begin{split}
\mathbf{H}=\int_{\mathbf{y}_0}^{y}e^{-e^{\mathbf{u}(x,t)}\cos\mathbf{v}(x,t)}e^{-ie^{\mathbf{u}(x,t)}\sin \mathbf{v}(x,t)}dt=\mathbf{H}_1+i\mathbf{H}_2
\end{split}
\end{equation}
and
\begin{equation}\label{Eq 76fiyt-2}
\begin{split}
\mathbf{H}_1&=\int_{\mathbf{y}_0}^{y}e^{-e^{\mathbf{u}(x,t)}\cos\mathbf{v}(x,t)}\cos\left(-e^{\mathbf{u}(x,t)}\sin \mathbf{v}(x,t)\right)dt,\\
\mathbf{H}_2&=\int_{\mathbf{y}_0}^{y}e^{-e^{\mathbf{u}(x,t)}\cos\mathbf{v}(x,t)}\sin\left(-e^{\mathbf{u}(x,t)}\sin \mathbf{v}(x,t)\right)dt.
\end{split}
\end{equation}
To estimate $|f(z)|$, we need to estimate $\mathbf{F}_0$ and $\mathbf{H}$ respectively. For each $k=0,1,\cdots,\mathbf{k}$, we let $V_k$ be the part of $V$ between the two points $\mathbf{z}_0$ and $\mathbf{z}_k$ and denote
\begin{equation}\label{Eq ye4evrfnwli-2}
\begin{split}
\mathbf{J}_k=\int_{V_{k}}e^{-U(z)}dz.
\end{split}
\end{equation}
Then, by the Cauchy integral theorem, for any two integers $k_1$ and $k_2$ such that $0\leq k_1<k_2\leq \mathbf{k}$, we have
\begin{equation}\label{Eq ye4evrfnwli-3}
\begin{split}
\mathbf{F}_{k_1}+\mathbf{J}_{k_2}-\mathbf{J}_{k_1}=\mathbf{G}_{k_{2}}-\mathbf{G}_{k_{1}}+\mathbf{F}_{k_2}.
\end{split}
\end{equation}
In particular, when $k_1=0$ and $k_2=\mathbf{k}$, we see that the value of $\mathbf{J}_\mathbf{k}$ is related to the values of $\mathbf{F}_0$, $\mathbf{F}_\mathbf{k}$ and $\mathbf{G}_\mathbf{k}$.

Moreover, for the integer $\mathbf{k}$, we may write $\mathbf{v}=\mathbf{v}(x,y)= 2\mathbf{k}\pi+\zeta$ with a number $\zeta\in[-\pi/2,3\pi/2)$. By the definition of $\mathbf{v}$, there is a curve $\Psi_{\mathbf{k},\zeta}$ on which
\begin{equation*}
\begin{split}
z=s+it, \quad \mathbf{v}(s,t)=2k\pi+\zeta, \quad s\in[x_0,\infty).
\end{split}
\end{equation*}
By similar arguments as before, we may suppose that the curve $\Omega$ intersects with the curve $\Psi_{\mathbf{k},\zeta}$ at the unique point $z_{\mathbf{k},\zeta}=x_{\mathbf{k},\zeta}+iy_{\mathbf{k},\zeta}$. Denote by $L_{\mathbf{k},\zeta}$ the part of $\Psi_{\mathbf{k},\zeta}$ between the point $z_{\mathbf{k},\zeta}$ and the point $z$ and by $\Omega_{\mathbf{k},\zeta}$ the part of $\Omega$ between the two points $z_{\mathbf{k}}$ and $z_{\mathbf{k},\zeta}$ respectively. We integrate $e^{-U(z)}$ along the curve $L_{\mathbf{k},\zeta}$ and the curve $\Omega_{\mathbf{k},\zeta}$ and denote
\begin{equation}\label{Eq 4giubjb-new-2}
\begin{split}
\mathbf{F}_{\mathbf{k},\zeta}=\int_{L_{\mathbf{k},\zeta}}e^{-U(z)}dz
\end{split}
\end{equation}
and
\begin{equation}\label{Eq 4giubjb-new-3}
\begin{split}
\mathbf{G}_{\mathbf{k},\zeta}=\int_{\Omega_{\mathbf{k},\zeta}}e^{-U(z)}dz.
\end{split}
\end{equation}
Let $V_{\mathbf{k},\zeta}$ be the part of $V$ between the two points $\mathbf{z}_\mathbf{k}$ and $z$ and denote
\begin{equation}\label{Eq 4giubjb-new-4}
\begin{split}
\mathbf{J}_{\mathbf{k},\zeta}=\int_{V_{\mathbf{k},\zeta}}e^{-U(z)}dz.
\end{split}
\end{equation}
Then, by the Cauchy integral theorem, we have
\begin{equation}\label{Eq 4giubjb-new-5}
\begin{split}
\mathbf{F}_{\mathbf{k}}+\mathbf{J}_{\mathbf{k},\zeta}=\mathbf{G}_{\mathbf{k},\zeta}+\mathbf{F}_{\mathbf{k},\zeta}.
\end{split}
\end{equation}
We see that the value of $\mathbf{J}_{\mathbf{k},\zeta}$ is related to the values of $\mathbf{F}_\mathbf{k}$, $\mathbf{F}_{\mathbf{k},\zeta}$ and $\mathbf{G}_{\mathbf{k},\zeta}$. By the above discussions, below we estimate $\mathbf{F}_{k}$, $\mathbf{G}_k$, $\mathbf{F}_{\mathbf{k},\zeta}$ and $\mathbf{G}_{\mathbf{k},\zeta}$, respectively.

\subsection{Estimates for $\mathbf{F}_{k}$, $\mathbf{G}_k$, $\mathbf{F}_{\mathbf{k},\zeta}$ and $\mathbf{G}_{\mathbf{k},\zeta}$}\label{Part III}

First, we estimate $\mathbf{F}_{k}$ defined in \eqref{Eq ye4evrfnwli-1} for each integer $k=0,1,\cdots,\mathbf{k}$. Since $L_k$ has the initial point $z_k=x_k+iy_k$ and the end point $\mathbf{z}_k=x+i\mathbf{y}_k$, along the curve $L_k$ we may write $\mathbf{F}_k$ in \eqref{Eq ye4evrfnwli-1} as
\begin{equation}\label{Eqcnzgsjhh}
\begin{split}
\mathbf{F}_k=\int_{x_k}^{x}e^{-e^{\tilde{\mathbf{u}}}}ds+i\int_{x_k}^{x}e^{-e^{\tilde{\mathbf{u}}}}dt=\int_{x_k}^{x}e^{-e^{\tilde{\mathbf{u}}}}ds+i\int_{x_k}^{x}e^{-e^{\tilde{\mathbf{u}}}}\frac{dt}{ds}ds.
\end{split}
\end{equation}
Along each curve $L_k$, we take the first derivative on both sides of the equation $\mathbf{v}(s,t)=2k\pi$ with respect to $s$ and obtain \eqref{Eq 6aren;f} and \eqref{Eq 6aren;f-1}. Since the two estimates in \eqref{Eq 77 a1 fu1aojr-3hilu-2giu-1} and \eqref{Eq 77 a1 fu1aojr-3hilu-2giu-2} hold along $\Omega$ and $\Phi$, this implies that $\tilde{\mathbf{u}}_t^2+\tilde{\mathbf{v}}_t^2= [1+\epsilon(t)]|nz^{n-1}|^2$ uniformly for all $z=s+it\in L_k^{l}$ and all $k=0,1,\cdots,\mathbf{k}$ such that $L_k^{l}$ is the part of $L_k$ between $\Omega$ and $\Phi$. For $z=s+it\in L_k^{r}$ and all $k=0,1,\cdots$ such that $L_k^{r}$ is the part of $L_k$ between $\Phi$ and $V$, we may replace the quantity $\epsilon(y)$ in \eqref{Eq 77 a1 fu1aojr-3hilu-2giu-1} and \eqref{Eq 77 a1 fu1aojr-3hilu-2giu-2} by $\epsilon(x)$ and also obtain $\tilde{\mathbf{u}}_t^2+\tilde{\mathbf{v}}_t^2=[1+\epsilon(s)]|nz^{n-1}|^2$. We conclude that there is a positive constant $\bar{\omega}>0$ such that $\tilde{\mathbf{u}}_t^2+\tilde{\mathbf{v}}_t^2\geq \bar{\omega}^2$ for all $z=s+it\in L_k$ and all $k=0,1,\cdots,\mathbf{k}$. Now, along the curve $L_k$, we take the first derivative of $\tilde{\mathbf{u}}$ with respect to $s$ and obtain
\begin{equation*}
\begin{split}
\frac{d\tilde{\mathbf{u}}}{ds}=\tilde{\mathbf{u}}_s+\tilde{\mathbf{u}}_t\frac{dt}{ds}.
\end{split}
\end{equation*}
By the Cauchy--Riemann equations we have $\mathbf{u}_x=\mathbf{v}_y$. Thus, together with the equation in \eqref{Eq 6aren;f-1}, we have
\begin{equation}\label{Eq qgagfqqa}
\begin{split}
\frac{d\tilde{\mathbf{u}}}{ds}=\frac{\tilde{\mathbf{u}}^2_t+\tilde{\mathbf{v}}^2_t}{\tilde{\mathbf{v}}_t}.
\end{split}
\end{equation}
Moreover, as noted before, for any $z\in A_{\omega}$ we may write the function $U(z)$ in \eqref{Eq 63hil} as
\begin{equation*}
\begin{split}
U(z)=\int_{z_0}^zS(t)e^{P(t)}dt=e^{\mathbf{u}+i\mathbf{v}}
\end{split}
\end{equation*}
and it is elementary to show that
\begin{equation*}
\begin{split}
\left|\frac{U'(z)}{U(z)}\right|=\left|\mathbf{u}_y+i\mathbf{v}_y\right|=\sqrt{\mathbf{u}_y^2+\mathbf{v}_y^2}.
\end{split}
\end{equation*}
Note that $\mathbf{u}^2_y+\mathbf{v}^2_y\geq |\mathbf{v}_y|^2$. Thus by \eqref{Eq qgagfqqa} and previous discussions we have $|d\tilde{\mathbf{u}}/ds|\geq \bar{\omega}$ for all $z=s+it\in L_k$ and all $k=0,1,\cdots,\mathbf{k}$. Then we have
\begin{equation}\label{Eqcnzgset-4}
\begin{split} \left|\int_{x_k}^{x}e^{-e^{\tilde{\mathbf{u}}}}ds\right|=\int_{\mathbf{u}(x_k,y_k)}^{\mathbf{u}(x,\mathbf{y}_k)}\left|e^{-e^{\tilde{\mathbf{u}}}}\frac{ds}{d\tilde{\mathbf{u}}}\right|d\tilde{\mathbf{u}}\leq \frac{1}{\bar{\omega}}\int_{\mathbf{u}(x_k,y_k)}^{\mathbf{u}(x,\mathbf{y}_k)}e^{-e^{\tilde{\mathbf{u}}}}d\tilde{\mathbf{u}}.
\end{split}
\end{equation}
When $n=1$, we see that $\tilde{\mathbf{u}}=s+\epsilon(s)$ along each curve $L_k$. When $n\geq 2$, by \eqref{Eq 77 a1 fu1aojr-3hilu-3} and \eqref{Eq 6aren;f-1} we see that $t$ decreases as $s\to\infty$ along each curve $L_k$ and, moreover, by the estimates in \eqref{Eq 77 a1 fu1aojr-3hilu-2giu-1}, \eqref{Eq 77 a1 fu1aojr-3hilu-2giu-2} and \eqref{Eq 66vfiytdyst-1} we see that $\tilde{\mathbf{u}}_s=\tilde{\mathbf{v}}_t>0$ and $\tilde{\mathbf{u}}_t<0$ along each curve $L_k$. In particular, we have the estimate for $\mathbf{u}$ along the ray $\Phi$ in \eqref{Eq 66vfiytdyst}. By replacing $y=x\tan\theta$ in \eqref{Eq 66vfiytdyst}, we see that $\tilde{\mathbf{u}}$ increases as $z=s+it\to\infty$ along each curve $L_k$. Thus, the value of $\mathbf{u}(s,\mathbf{y}_{k})$ is always increasing and tends to $\infty$ as $s\to\infty$ along each curve $L_k$ in both of the two cases $n=1$ and $n\geq 2$. Since $e^{\mathbf{u}(x_k,y_k)}=\omega$ for all $k=0,1,\cdots,\mathbf{k}$, we may use similar arguments as in \eqref{Eq 4nalr-1} to deduce from \eqref{Eqcnzgset-4} that
\begin{equation*}
\begin{split}
\int_{x_k}^{x}e^{-e^{\tilde{\mathbf{u}}}}ds=d_{k,1}+O\left(\mathbf{u}(x,\mathbf{y}_k)e^{-e^{\mathbf{u}(x,\mathbf{y}_k)}}\right)=d_{k,1}+o(1), \quad x\to\infty
\end{split}
\end{equation*}
for some real constant $d_{k,1}$. For the last integral in \eqref{Eqcnzgsjhh}, we note that
\begin{equation}\label{Eq vfudt}
\begin{split}
\frac{dt}{ds}\frac{ds}{d\tilde{\mathbf{u}}}=\frac{\tilde{\mathbf{u}}_t}{\tilde{\mathbf{v}}_t}\frac{\tilde{\mathbf{v}}_t}{\tilde{\mathbf{u}}^2_t+\tilde{\mathbf{v}}^2_t}
=\frac{\tilde{\mathbf{u}}_t}{\tilde{\mathbf{u}}^2_t+\tilde{\mathbf{v}}_t^2}.
\end{split}
\end{equation}
Then, by similar arguments as above together with \eqref{Eq vfudt}, we get
\begin{equation}\label{Eqcnzgset-4-dyt}
\begin{split}
\left|\int_{x_k}^{x}e^{-e^{\tilde{\mathbf{u}}}}\frac{dt}{ds}ds\right|=\int_{\mathbf{u}(x_k,y_k)}^{\mathbf{u}(x,\mathbf{y}_k)}\left|e^{-e^{\tilde{\mathbf{u}}}}\frac{dt}{ds}\frac{ds}{d\tilde{\mathbf{u}}}\right|d\tilde{\mathbf{u}}\leq \frac{1}{\bar{\omega}}\int_{\mathbf{u}(x_k,y_k)}^{\mathbf{u}(x,\mathbf{y}_k)}e^{-e^{\tilde{\mathbf{u}}}}d\tilde{\mathbf{u}}
\end{split}
\end{equation}
and thus
\begin{equation*}
\begin{split}
\int_{x_k}^{x}e^{-e^{\tilde{\mathbf{u}}}}\frac{dt}{ds}ds=d_{k,2}+O\left(\mathbf{u}(x,\mathbf{y}_k)e^{-e^{\mathbf{u}(x,\mathbf{y}_k)}}\right)=d_{k,2}+o(1), \quad x\to\infty
\end{split}
\end{equation*}
for some real constant $d_{k,2}$. By the definition of $\mathbf{F}_k$ in \eqref{Eq ye4evrfnwli-1} and the expression of $\mathbf{u}$ along the ray $\Phi$ in \eqref{Eq 66vfiytdyst}, we conclude that
\begin{equation}\label{Eqcnzgset-4gui}
\begin{split}
\mathbf{F}_k=d_k+O\left(\mathbf{u}(x,\mathbf{y}_k)e^{-e^{\mathbf{u}(x,\mathbf{y}_k)}}\right)=d_{k}+o(1), \quad x\to\infty
\end{split}
\end{equation}
for some constant $d_k$ along each curve $L_k$. Moreover, since $e^{\mathbf{u}(x_k,y_k)}=\omega$ and $\tilde{\mathbf{u}}_t^2+\tilde{\mathbf{v}}_t^2\geq \bar{\omega}^2$ for all $z=s+it\in L_k$ and all $k=0,1,\cdots,\mathbf{k}$, we see from the inequalities in \eqref{Eqcnzgset-4} and \eqref{Eqcnzgset-4-dyt} that $|d_k|$ is actually uniformly bounded for all integers $k\geq 0$.

Second, we estimate $\mathbf{F}_{\mathbf{k},\zeta}$ defined in \eqref{Eq 4giubjb-new-2} for the case when $\zeta\in[\pi/2+\varepsilon,3\pi/2-\varepsilon]$. Since $L_{\mathbf{k},\zeta}$ has the initial point $z_{\mathbf{k},\zeta}=x_{\mathbf{k},\zeta}+iy_{\mathbf{k},\zeta}$ and the end point $z=x+iy$, along the curve $L_{\mathbf{k},\zeta}$ we may write $\mathbf{F}_{\mathbf{k},\zeta}$ in \eqref{Eq 4giubjb-new-2} as
\begin{equation*}
\begin{split}
\mathbf{F}_{\mathbf{k},\zeta}
&=\int_{x_{\mathbf{k},\zeta}}^{x}e^{-e^{\tilde{\mathbf{u}}}\cos\zeta}\cos\left(-e^{\tilde{\mathbf{u}}}\sin\zeta\right)ds+i\int_{x_{\mathbf{k},\zeta}}^{x}e^{-e^{\tilde{\mathbf{u}}}\cos\zeta}\sin\left(-e^{\tilde{\mathbf{u}}}\sin\zeta\right)\frac{dt}{ds}ds.
\end{split}
\end{equation*}
Along the curve $L_{\mathbf{k},\zeta}$, we take the first derivative on both sides of the equation $\mathbf{v}(s,t)=2k\pi+\zeta$ with respect to $s$ and obtain \eqref{Eq 6aren;f} and \eqref{Eq 6aren;f-1}. Then by \eqref{Eq qgagfqqa} we see that $d\tilde{\mathbf{u}}/ds>0$ and thus $\tilde{\mathbf{u}}$ is increasing as $z\to\infty$ along the curve $L_{\mathbf{k},\zeta}$. Since $\cos\tilde{\mathbf{v}}=\cos\zeta<0$, by \eqref{Eq 6aren;f-1} together with the estimates in \eqref{Eq 77 a1 fu1aojr-3hilu-2} and \eqref{Eq 77 a1 fu1aojr-3hilu-3} for $\mathbf{u}_y/\mathbf{v}_y$ and the relation in \eqref{Eq 6aren;f-1}, it is easy to see that
\begin{equation}\label{Eq 4giubjb-new-8}
\begin{split}
\left|\text{Im}(F_{\mathbf{k},\zeta})\right|&=\left|\int_{x_{\mathbf{k},\zeta}}^{x}e^{-e^{\tilde{\mathbf{u}}}\cos\zeta}\sin\left(-e^{\tilde{\mathbf{u}}}\cos\zeta\right)\frac{dt}{ds}ds\right|\leq \kappa[1+\epsilon(x)]xe^{-e^{\mathbf{u}}\cos\zeta}
\end{split}
\end{equation}
and
\begin{equation}\label{Eq 4giubjb-new-9}
\begin{split}
\left|\text{Re}(\mathbf{F}_{\mathbf{k},\zeta})\right|=\left|\int_{x_{\mathbf{k},\zeta}}^{x}e^{-e^{\tilde{\mathbf{u}}}\cos\zeta}\cos\left(-e^{\tilde{\mathbf{u}}}\cos\zeta\right)ds\right|\leq xe^{-e^{\mathbf{u}}\cos\zeta}.
\end{split}
\end{equation}

Third, we estimate $\mathbf{G}_{\mathbf{k}}$ defined in \eqref{Eq ye4evrfnwli-1-1}. Since the curve $\Omega_{k}$ has the initial point $z_0=x_0+iy_0$ and the end point $z_k=x_k+iy_k$, we have
\begin{equation}\label{Eq ye4evr-pre}
\begin{split}
\mathbf{G}_{\mathbf{k}}=\sum_{k=1}^{\mathbf{k}}(\mathbf{G}_{k}-\mathbf{G}_{k-1}).
\end{split}
\end{equation}
Moreover, along the curve $\Omega_{k}$, $k=0,1,\cdots,\mathbf{k}$, we may write $\mathbf{G}_{k}$ in \eqref{Eq ye4evrfnwli-1-1} as
\begin{equation}\label{Eq ye4evr}
\begin{split}
\mathbf{G}_{k}=\int_{y_0}^{y_{k}}e^{-\omega\cos\tilde{\mathbf{v}}}e^{-i\sin \omega\tilde{\mathbf{v}}}(ds+idt)=\text{Re}(\mathbf{G}_{k})+i\text{Im}(\mathbf{G}_{k}),
\end{split}
\end{equation}
where
\begin{equation*}
\begin{split}
\text{Re}(\mathbf{G}_{k})&=\int_{y_0}^{y_{k}}\left[e^{-\omega\cos\tilde{\mathbf{v}}}\cos(\omega\sin \tilde{\mathbf{v}})ds+e^{-\omega\cos\tilde{\mathbf{v}}}\sin(\omega\sin \tilde{\mathbf{v}})dt\right],\\
\text{Im}(\mathbf{G}_{k})&=\int_{y_0}^{y_{k}}\left[e^{-\omega\cos\tilde{\mathbf{v}}}\cos(\omega\sin \tilde{\mathbf{v}})dt-e^{-\omega\cos\tilde{\mathbf{v}}}\sin(\omega\sin \tilde{\mathbf{v}})ds\right].
\end{split}
\end{equation*}
By the definition of $\Omega$, $e^{\mathbf{u}}=\omega$ for some positive constant $\omega$ for $z\in \Omega$. Then we see that $x=\epsilon(y)$ when $n=1$ and $y=[1+\epsilon(y)][\tan(\pi/2n)]x$ when $n\geq 2$. By taking the first derivative on both sides of the equation $e^{\tilde{\mathbf{u}}}=\omega$ with respect to $t$ along the curve $\Omega$, we have
\begin{equation*}
\begin{split}
\tilde{\mathbf{u}}_s\frac{ds}{dt}+\tilde{\mathbf{u}}_t=0.
\end{split}
\end{equation*}
Since $\mathbf{u}$ and $\mathbf{v}$ are the real and imaginary parts of an analytic function, we have by the Cauchy--Riemann equations that $\mathbf{u}_x=\mathbf{v}_y$. It follows from the above equation that
\begin{equation}\label{Eq qgagre}
\begin{split}
\frac{ds}{dt}=-\frac{\tilde{\mathbf{u}}_t}{\tilde{\mathbf{v}}_t}.
\end{split}
\end{equation}
Along the curve $\Omega$, we have the two estimates in \eqref{Eq 77 a1 fu1aojr-3hilu-2} and \eqref{Eq 77 a1 fu1aojr-3hilu-3} for $\mathbf{u}_y/\mathbf{v}_y$. In particular, $ds/dt=\epsilon(t)$ when $n=1$. Then it follows by the Cauchy--Riemann equations that $\mathbf{v}_x=-\mathbf{u}_y$ and so, together with the relation in \eqref{Eq qgagre},
\begin{equation}\label{Eq qgagre-1}
\begin{split}
\frac{d\tilde{\mathbf{v}}}{dt}=\tilde{\mathbf{v}}_s\frac{ds}{dt}+\tilde{\mathbf{v}}_t=\frac{\tilde{\mathbf{u}}^2_t+\tilde{\mathbf{v}}^2_t}{\tilde{\mathbf{v}}_t}.
\end{split}
\end{equation}
Recall the notations $\kappa$ and $\kappa_1$ in \eqref{Eq 77 notation}. Then by the relations \eqref{Eq 77 a1 fu1aojr-3hilu-2} and \eqref{Eq 77 a1 fu1aojr-3hilu-3} and also the expressions for $\mathbf{v}_y$ in \eqref{Eq 77 a1 fu1aojr-3hilu-1-fu}, we may write uniformly as
\begin{equation}\label{Eq qgagre-4}
\begin{split}
\frac{dt}{d\tilde{\mathbf{v}}}=\frac{1}{\tilde{\mathbf{v}}_t\left(1+\frac{\tilde{\mathbf{u}}^2_t}{\tilde{\mathbf{v}}^2_t}\right)}=\frac{1+\epsilon(t)}{\tilde{\mathbf{v}}_t(1+\kappa^2)}
=\frac{1+\epsilon(t)}{\kappa_1(1+\kappa^2)t^{n-1}}.
\end{split}
\end{equation}

Here we shall mainly look at $\text{Im}(\mathbf{G}_{\mathbf{k}})$. All the arguments below used to estimate $\text{Im}(\mathbf{G}_{\mathbf{k}})$ also apply to $\text{Re}(\mathbf{G}_{\mathbf{k}})$. We write $s=s(t)$. Then, by \eqref{Eq ye4evr-pre} and \eqref{Eq ye4evr}, we have
\begin{equation}\label{Eq ye4evr-2}
\begin{split}
\text{Im}(\mathbf{G}_{\mathbf{k}})=\sum_{k=1}^{\mathbf{k}}\text{Im}(\mathbf{G}_{k}-\mathbf{G}_{k-1}),
\end{split}
\end{equation}
where
\begin{equation}\label{Eq ye4evr-3}
\begin{split}
\text{Im}(\mathbf{G}_{k}-\mathbf{G}_{k-1})=\mathbf{G}_{k,1}-\mathbf{G}_{k,2}
\end{split}
\end{equation}
and
\begin{equation*}
\begin{split}
\mathbf{G}_{k,1}&=\int_{y_{k-1}}^{y_{k}}e^{-\omega\cos\tilde{\mathbf{v}}}\cos(\omega\sin \tilde{\mathbf{v}})dt=\int_{2(k-1)\pi}^{2k\pi}e^{-\omega\cos\tilde{\mathbf{v}}}\cos(\omega\sin \tilde{\mathbf{v}})\frac{dt}{d\tilde{\mathbf{v}}}d\tilde{\mathbf{v}},\\
\mathbf{G}_{k,2}&=\int_{y_{k-1}}^{y_{k}}e^{-\omega\cos\tilde{\mathbf{v}}}\sin(\omega\sin \tilde{\mathbf{v}})ds=\int_{2(k-1)\pi}^{2k\pi}e^{-\omega\cos\tilde{\mathbf{v}}}\sin(\omega\sin \tilde{\mathbf{v}})\frac{ds}{dt}\frac{dt}{d\tilde{\mathbf{v}}}d\tilde{\mathbf{v}}.
\end{split}
\end{equation*}
To estimate $\mathbf{G}_{k,1}$, we let $E_k=[2(k-1)\pi, 2k\pi]$. Then, for each $k=1,2,\cdots,\mathbf{k}$, we divide the set $E_k$ into two parts: $E_k=E_{k,1}\cup E_{k,2}$, where
\begin{equation*}
\begin{split}
E_{k,1}&=\left\{\tilde{\mathbf{v}}\in E_k: \cos\left(\omega\sin\tilde{\mathbf{v}}\right)\geq 0\right\},\\
E_{k,2}&=\left\{\tilde{\mathbf{v}}\in E_k: \cos\left(\omega\sin\tilde{\mathbf{v}}\right)<0\right\}.
\end{split}
\end{equation*}
It follows that
\begin{equation*}
\begin{split}
\mathbf{G}_{k,1}&=\int_{E_{k,1}}e^{-\omega\cos\tilde{\mathbf{v}}}\cos(\omega\sin \tilde{\mathbf{v}})\frac{dt}{d\tilde{\mathbf{v}}}d\tilde{\mathbf{v}}+\int_{E_{k,2}}e^{-\omega\cos\tilde{\mathbf{v}}}\cos(\omega\sin \tilde{\mathbf{v}})\frac{dt}{d\tilde{\mathbf{v}}}d\tilde{\mathbf{v}}.
\end{split}
\end{equation*}
Let $y_{k,i}\in[y_{k-1},y_{k}]$, $i=1,2,\cdots$ be any point. Denote
$\mathbf{v}^{k,i}_y=\kappa_1(1+\kappa^2)y_{k,i}^{n-1}$ and  $\mathbf{v}^{k}_y=\kappa_1(1+\kappa^2)y_{k}^{n-1}$ for simplicity. We see that $\mathbf{v}^{k,i}_y=[1+\epsilon(y_k)]\mathbf{v}^{k,j}_y$ for any two points $y_{k,i},y_{k,j}\in[y_{k-1},y_{k}]$. By \eqref{Eq qgagre-1}, \eqref{Eq qgagre-4} and the first mean value theorem of integral together with the first identity in \eqref{Eq 70-cor4}, when $y_k$ is large, we have
\begin{equation*}
\begin{split}
\mathbf{G}_{k,1}&=\frac{1+\epsilon(y_k)}{\mathbf{v}^{k,1}_y}\int_{E_{k,1}}e^{-\omega\cos\tilde{\mathbf{v}}}\cos(\omega\sin \tilde{\mathbf{v}})d\tilde{\mathbf{v}}+\frac{1+\epsilon(y_k)}{\mathbf{v}^{k,2}_y}\int_{E_{k,2}}e^{-\omega\cos\tilde{\mathbf{v}}}\cos(\omega\sin \tilde{\mathbf{v}})d\tilde{\mathbf{v}}\\
&=\frac{1}{\mathbf{v}^{k,1}_y}\int_{2(k-1)\pi}^{2k\pi}e^{-\omega\cos\tilde{\mathbf{v}}}\cos(\omega\sin \tilde{\mathbf{v}})d\tilde{\mathbf{v}}+\frac{\epsilon(y_k)\cdot O(1)}{\mathbf{v}^{k,1}_y} =\frac{1+\epsilon(y_k)}{\mathbf{v}^{k}_y}\cdot 2\pi
\end{split}
\end{equation*}
for some $y_{k,1},y_{k,2}\in[y_{k-1},y_{k}]$. To estimate $\mathbf{G}_{k,2}$, for each $k=1,2,\cdots,\mathbf{k}$, we divide the set $E_k=[2(k-1)\pi, 2k\pi]$ into two parts: $E_k=E_{k,3}\cup E_{k,4}$, where
\begin{equation*}
\begin{split}
E_{k,3}&=\left\{\tilde{\mathbf{v}}\in E_k: \sin\left(\omega\sin\tilde{\mathbf{v}}\right)\geq 0\right\},\\
E_{k,4}&=\left\{\tilde{\mathbf{v}}\in E_k: \sin\left(\omega\sin\tilde{\mathbf{v}}\right)<0\right\}.
\end{split}
\end{equation*}
Then we may write
\begin{equation*}
\begin{split}
\mathbf{G}_{k,2}
&=\int_{E_{k,3}}e^{-\omega\cos\tilde{\mathbf{v}}}\sin(\omega\sin \tilde{\mathbf{v}})\frac{d}{dt}\frac{dt}{d\tilde{\mathbf{v}}}d\tilde{\mathbf{v}}+\int_{E_{k,4}}e^{-\omega\cos\mathbf{v}}\sin(\omega\sin \tilde{\mathbf{v}})\frac{ds}{dt}\frac{dt}{d\tilde{\mathbf{v}}}d\tilde{\mathbf{v}}
\end{split}
\end{equation*}
and, again, it follows by the first mean value theorem of integral together with the second identity in \eqref{Eq 70-cor4} and also the relations in \eqref{Eq 77 a1 fu1aojr-3hilu-2}, \eqref{Eq 77 a1 fu1aojr-3hilu-3}, \eqref{Eq qgagre} and \eqref{Eq qgagre-4} that
\begin{equation*}
\begin{split}
\mathbf{G}_{k,2}&=\frac{[1+\epsilon(y_k)]\kappa}{\mathbf{v}^{k,4}_y}\int_{E_{k,3}}e^{-\omega\cos\tilde{\mathbf{v}}}\sin(\omega\sin \tilde{\mathbf{v}})d\tilde{\mathbf{v}}+\frac{[1+\epsilon(y_k)]\kappa}{\mathbf{v}^{k,5}_y}\int_{E_{k,4}}e^{-\omega\cos\tilde{\mathbf{v}}}\sin(\omega\sin \tilde{\mathbf{v}})d\tilde{\mathbf{v}}\\
&=\frac{\kappa}{\mathbf{v}^{k,4}_y}\int_{2(k-1)\pi}^{2k\pi}e^{-\omega\cos\tilde{\mathbf{v}}}\sin(\omega\sin \tilde{\mathbf{v}})d\tilde{\mathbf{v}}+\frac{\epsilon(y_k)\kappa\cdot O(1)}{\mathbf{v}^{k,4}_y} =\frac{\epsilon(y_k)}{\mathbf{v}^{k}_y}
\end{split}
\end{equation*}
for some $y_{k,3},y_{k,4}\in[y_{k-1},y_{k}]$. Therefore, for each $k=1,2,\cdots,\mathbf{k}$, by the above estimates for $\mathbf{G}_{k,1}$ and $\mathbf{G}_{k,2}$ we always have
\begin{equation}\label{Eq 7giuivftrd}
\begin{split}
\text{Im}(\mathbf{G}_{k}-\mathbf{G}_{k-1})=\mathbf{G}_{k,1}-\mathbf{G}_{k,2}=\frac{[1+\epsilon(y_k)]2\pi}{\mathbf{v}^{k}_y}=\frac{[1+\epsilon(y_k)]2\pi}{\kappa_1(1+\kappa^2)y_k^{n-1}}.
\end{split}
\end{equation}
Note that $y_k<y_{\mathbf{k}}$ for $k=1,\cdots,\mathbf{k}-1$ and thus $\epsilon(y_k)=\epsilon(y_{\mathbf{k}})$ for $k=1,\cdots,\mathbf{k}-1$. Then, by summarising the above calculations for $\text{Im}(\mathbf{G}_{k}-\mathbf{G}_{k-1})$ together with the relations in \eqref{Eq ye4evr-2} and \eqref{Eq ye4evr-3}, we have
\begin{equation}\label{Eq 70-corvjh}
\begin{split}
\text{Im}(\mathbf{G}_{\mathbf{k}})=\sum_{k=1}^{\mathbf{k}}\frac{[1+\epsilon(y_k)]2\pi}{\kappa_1(1+\kappa^2)y_k^{n-1}}
\geq \frac{[1+\epsilon(y_\mathbf{k})]2\pi\mathbf{k}}{\kappa_1(1+\kappa^2)y_\mathbf{k}^{n-1}}.
\end{split}
\end{equation}
Here we point out the relation between $y_k$ and $\mathbf{y}_{k}$. Recall that the curve $\mathbf{V}$ intersects with the curve $L_k$ at the point $\mathbf{z}_k=x+i\mathbf{y}_k$ and the curve $\Omega$ intersects with the curve $L_k$ at the points $z_k=x_k+iy_k$. Then, by the expression of $\mathbf{v}$ along the ray $\Phi$ and the curve $\Omega$ in \eqref{Eq 66vfiytdystbk} and the definition of $L_k$, we have
\begin{equation*}
\begin{split}
[1+\epsilon(y_k)]y_k^n=\frac{\sin^n\vartheta}{\sin^n\theta}\frac{\sin n\theta}{\sin n\vartheta}\mathbf{y}_k^n[1+\epsilon(\mathbf{y}_k)],
\end{split}
\end{equation*}
which implies that
\begin{equation}\label{Eq 66vfiytdystgiy-1}
\begin{split}
y_k=\frac{\sin\vartheta}{\sin\theta}\sqrt[n]{\frac{\sin n\theta}{\sin n\vartheta}}[1+\epsilon(\mathbf{y}_{k})]\mathbf{y}_k.
\end{split}
\end{equation}
Note that $\mathbf{v}(x,\mathbf{y}_{\mathbf{k}})=2\pi\mathbf{k}[1+\epsilon(\mathbf{y}_{\mathbf{k}})]$ and also that $\mathbf{v}_y(x,y_{\mathbf{k}})=\kappa_1y_{\mathbf{k}}^{n-1}[1+\epsilon(y_{\mathbf{k}})]$ along the curve $\Omega$ by \eqref{Eq 77 a1 fu1aojr-3hilu-1-fu}. Recall the notations $\kappa$ and $\kappa_1$ in \eqref{Eq 77 notation}. Then by the estimate for $\mathbf{v}$ in \eqref{Eq 66vfiytdystbk} along the ray $\Phi$, we have from \eqref{Eq 70-corvjh} that
\begin{equation}\label{Eq 70-corafeaf4}
\begin{split}
\text{Im}(\mathbf{G}_{\mathbf{k}})\geq\frac{\sqrt[n]{\sin n\theta}\cos(n-1)\vartheta}{n\sin\theta}(\sin n\vartheta)^{\frac{n-1}{n}}[1+\epsilon(\mathbf{y}_{\mathbf{k}})]\mathbf{y}_{\mathbf{k}}
\end{split}
\end{equation}
in both of the two cases $n=1$ and $n\geq 2$. Recall that $\vartheta=[1+\epsilon(y)+O(y^{-n}\ln y)]\pi/2n$ along the curve $\Omega$ and thus $\vartheta=[1+\epsilon(\mathbf{y}_{\mathbf{k}})+O(\mathbf{y}_{\mathbf{k}}^{-n}\ln \mathbf{y}_{\mathbf{k}})]\pi/2n$ by the relation in \eqref{Eq 66vfiytdystgiy-1} and thus
\begin{equation}\label{Eq 70-blkgl-1}
\begin{split}
\sin n\vartheta=1+\epsilon(\mathbf{y}_{\mathbf{k}})+O\left(\mathbf{y}_{\mathbf{k}}^{-n}\ln \mathbf{y}_{\mathbf{k}}\right), \quad \mathbf{k}\to\infty.
\end{split}
\end{equation}
On the other hand, using exactly the same arguments, by \eqref{Eq ye4evr-pre} and \eqref{Eq ye4evr} together with the two derivatives in \eqref{Eq qgagre} and \eqref{Eq qgagre-1} we may also obtain
\begin{equation}\label{Eq 7giuivftrd-1}
\begin{split}
\text{Re}(\mathbf{G}_{k}-\mathbf{G}_{k-1})=\frac{\kappa[1+\epsilon(y_k)]2\pi}{\mathbf{v}^{k}_y}=\frac{\kappa[1+\epsilon(y_k)]2\pi}{\kappa_1(1+\kappa^2)y_k^{n-1}},
\end{split}
\end{equation}
which then yields
\begin{equation}\label{Eq 70-corafeaf4giy}
\begin{split}
\text{Re}(\mathbf{G}_{\mathbf{k}})\geq\frac{\sqrt[n]{\sin n\theta}\sin(n-1)\vartheta}{n\sin\theta}(\sin n\vartheta)^{\frac{n-1}{n}}[1+\epsilon(\mathbf{y}_{\mathbf{k}})]\mathbf{y}_{\mathbf{k}}
\end{split}
\end{equation}
in both of the two cases $n=1$ and $n\geq 2$. In particular, when $n=1$, by \eqref{Eq 77 a1 fu1aojr-3hilu-2} we will obtain that $\text{Re}(\mathbf{G}_{\mathbf{k}})=\epsilon(\mathbf{y}_{\mathbf{k}})\mathbf{y}_{\mathbf{k}}=O(1)$ from the above calculations.

Finally, we estimate $\mathbf{G}_{\mathbf{k},\zeta}$ defined in \eqref{Eq 4giubjb-new-3}. Since the curve $\Omega_{\mathbf{k},\zeta}$ has the initial point $z_\mathbf{k}=x_\mathbf{k}+iy_\mathbf{k}$ and the end point $z_{\mathbf{k},\zeta}=x_{\mathbf{k},\zeta}+iy_{\mathbf{k},\zeta}$, we may write $\mathbf{G}_{\mathbf{k},\zeta}$ in \eqref{Eq 4giubjb-new-2} as
\begin{equation*}
\begin{split}
\mathbf{G}_{\mathbf{k},\zeta}=\int_{y_\mathbf{k}}^{y_{\mathbf{k},\zeta}}e^{-\omega\cos\tilde{\mathbf{v}}}e^{-i\sin \omega\tilde{\mathbf{v}}}(ds+idt)=\text{Re}(\mathbf{G}_{\mathbf{k},\zeta})+i\text{Im}(\mathbf{G}_{\mathbf{k},\zeta}),
\end{split}
\end{equation*}
where
\begin{equation*}
\begin{split}
\text{Re}(\mathbf{G}_{\mathbf{k},\zeta})&=\int_{y_\mathbf{k}}^{y_{\mathbf{k},\zeta}}\left[e^{-\omega\cos\tilde{\mathbf{v}}}\cos(\omega\sin \tilde{\mathbf{v}})ds+e^{-\omega\cos\tilde{\mathbf{v}}}\sin(\omega\sin \tilde{\mathbf{v}})dt\right],\\
\text{Im}(\mathbf{G}_{\mathbf{k},\zeta})&=\int_{y_\mathbf{k}}^{y_{\mathbf{k},\zeta}}\left[e^{-\omega\cos\tilde{\mathbf{v}}}\cos(\omega\sin \tilde{\mathbf{v}})dt-e^{-\omega\cos\tilde{\mathbf{v}}}\sin(\omega\sin \tilde{\mathbf{v}})ds\right].
\end{split}
\end{equation*}
By similar arguments as before, it is easy to see that
\begin{equation}\label{Eq 4giubjb-new-11}
\begin{split}
\left|\text{Im}(\mathbf{G}_{\mathbf{k},\zeta})\right|\leq \frac{[1+\epsilon(y_\mathbf{k})]2\pi}{\mathbf{v}^{\mathbf{k}}_y}=\frac{[1+\epsilon(y_\mathbf{k})]2\pi}{\kappa_1(1+\kappa^2)y_\mathbf{k}^{n-1}}e^{\omega}
\end{split}
\end{equation}
and
\begin{equation}\label{Eq 4giubjb-new-12}
\begin{split}
\left|\text{Re}(\mathbf{G}_{\mathbf{k},\zeta})\right|\leq\frac{[1+\epsilon(y_\mathbf{k})]2\pi}{\mathbf{v}^{\mathbf{k}}_y}=\frac{\kappa[1+\epsilon(y_\mathbf{k})]2\pi}{\kappa_1(1+\kappa^2)y_\mathbf{k}^{n-1}}e^{\omega}
\end{split}
\end{equation}
in both of the two cases $n=1$ and $n\geq 2$. Thus $\text{Im}(\mathbf{G}_{\mathbf{k},\zeta})=O(1)$ and $\text{Re}(\mathbf{G}_{\mathbf{k},\zeta})=O(1)$. With the above estimates for $\mathbf{F}_{k}$, $\mathbf{G}_k$, $\mathbf{F}_{\mathbf{k},\zeta}$ and $\mathbf{G}_{\mathbf{k},\zeta}$, now we are ready to estimate $\mathbf{H}$ in \eqref{Eq 76fiyt-1} using the relations in \eqref{Eq ye4evrfnwli-3} and \eqref{Eq 4giubjb-new-5}.

\subsection{Completion of the proof}\label{Part IV}

For clarity, in this subsection we shall use the notations $\bar{\mathbf{u}}=\mathbf{u}(x,t)$ and $\bar{\mathbf{v}}=\mathbf{v}(x,t)$ which differ from $\mathbf{u}=\mathbf{u}(x,y)$ and $\mathbf{v}=\mathbf{v}(x,y)$ by the variable $t$. For the point $z=x+iy\in \Phi$, recall that $\mathbf{v}=\mathbf{v}(x,y)=2\mathbf{k}\pi+\zeta$ with an integer $\mathbf{k}$ and a number $\zeta\in[-\pi/2,3\pi/2)$. We choose $\mathbf{k}$ to be large. Then we may write $\mathbf{H}_1$ and $\mathbf{H}_2$ in \eqref{Eq 76fiyt-2} as
\begin{equation}\label{Eqcnzgset}
\begin{split}
\mathbf{H}_1=\mathbf{I}_{11}+\mathbf{I}_{12},\quad \mathbf{H}_2=\mathbf{I}_{21}+\mathbf{I}_{22},
\end{split}
\end{equation}
where
\begin{equation}\label{Eqcnzgset-1}
\begin{split}
\mathbf{I}_{11}&=\int_{\mathbf{y}_0}^{\mathbf{y}_{\mathbf{k}}}e^{-e^{\bar{\mathbf{u}}}\cos \bar{\mathbf{v}}}\cos\left(-e^{\bar{\mathbf{u}}}\sin\bar{\mathbf{v}}\right)dt,\\
\mathbf{I}_{12}&=\int_{\mathbf{y}_{\mathbf{k}}}^{y}e^{-e^{\bar{\mathbf{u}}}\cos\bar{\mathbf{v}}}\cos\left(-e^{\bar{\mathbf{u}}}\sin\bar{\mathbf{v}}\right)dt
\end{split}
\end{equation}
and
\begin{equation}\label{Eqcnzgset-qari-2}
\begin{split}
\mathbf{I}_{21}&=\int_{\mathbf{y}_0}^{\mathbf{y}_{\mathbf{k}}}e^{-e^{\bar{\mathbf{u}}}\cos \bar{\mathbf{v}}}\sin\left(-e^{\bar{\mathbf{u}}}\sin\bar{\mathbf{v}}\right)dt,\\
\mathbf{I}_{22}&=\int_{\mathbf{y}_{\mathbf{k}}}^{y}e^{-e^{\bar{\mathbf{u}}}\cos\bar{\mathbf{v}}}\sin\left(-e^{\bar{\mathbf{u}}}\sin\bar{\mathbf{v}}\right)dt.
\end{split}
\end{equation}
Below we estimate $\mathbf{I}_{11}$, $\mathbf{I}_{12}$, $\mathbf{I}_{21}$ and $\mathbf{I}_{22}$ respectively.

Consider first $\mathbf{v}=2\mathbf{k}\pi+\zeta$ with a number $\zeta\in[-\pi/2+\varepsilon,\pi/2-\varepsilon]$. For $\mathbf{I}_{11}$, we see by the definition of $\mathbf{J}_k$ in \eqref{Eq ye4evrfnwli-2} that $\mathbf{I}_{11}=\text{Im}(\mathbf{J}_{k})$. By taking the imaginary parts on both sides of equation \eqref{Eq ye4evrfnwli-3}, we find
\begin{equation*}
\begin{split}
\text{Im}(\mathbf{J}_{k_2}-\mathbf{J}_{k_1})=\text{Im}(\mathbf{G}_{k_{2}}-\mathbf{G}_{k_{1}})+\text{Im}(\mathbf{F}_{k_2})-\text{Im}(\mathbf{F}_{k_1}).
\end{split}
\end{equation*}
In particular, when $k_1=0$ and $k_2=\mathbf{k}$, by the estimates in \eqref{Eqcnzgset-4gui}, we have
\begin{equation}\label{Eqcnzgset-5}
\begin{split}
\mathbf{I}_{11}=\text{Im}(\mathbf{J}_{\mathbf{k}})=\text{Im}(\mathbf{G}_{\mathbf{k}})+\text{Im}(d_{\mathbf{k}}-d_{0})+o(1), \quad x\to\infty.
\end{split}
\end{equation}
Then, by the inequality in \eqref{Eq 70-corafeaf4}, we obtain from \eqref{Eqcnzgset-5} that
\begin{equation}\label{Eqcnzgset-6}
\begin{split}
\left|\mathbf{I}_{11}\right|\geq\frac{\sqrt[n]{\sin n\theta}\cos(n-1)\vartheta}{n\sin\theta}(\sin n\vartheta)^{\frac{n-1}{n}}[1+\epsilon(\mathbf{y}_{\mathbf{k}})]\mathbf{y}_{\mathbf{k}}.
\end{split}
\end{equation}
For $\mathbf{I}_{12}$, since $\mathbf{v}_y=\iota_1y^{n-1}[1+\epsilon(y)]$ with $\iota_1=n\cos (n-1)\theta/(\sin\theta)^{n-1}$ along the ray $\Phi$ by \eqref{Eq 77 a1 fu1aojr-3hilu-2giu-2} and since $d\bar{\mathbf{v}}/dt=\bar{\mathbf{v}}_t$, by similar arguments as for estimating $\mathbf{G}_{k,1}$ and $\mathbf{G}_{k,2}$ in previous subsection, we have
\begin{equation}\label{Eqcnzgset-7}
\begin{split}
\left|\mathbf{I}_{12}\right|&=\left|\int_{2\mathbf{k}\pi}^{2\mathbf{k}\pi+\zeta}e^{-e^{\bar{\mathbf{u}}}\cos\bar{\mathbf{v}}}\cos\left(-e^{\bar{\mathbf{u}}}\sin\bar{\mathbf{v}}\right)\frac{dt}{d\bar{\mathbf{v}}}d\bar{\mathbf{v}}\right|\\
&\leq \left|\frac{1+\epsilon(\mathbf{y}_\mathbf{k})}{\iota_1\mathbf{y}_{\mathbf{k}}^{n-1}}\right|\left|\int_{2\mathbf{k}\pi}^{2\mathbf{k}\pi+\zeta}\left|e^{-e^{\bar{\mathbf{u}}}\cos\bar{\mathbf{v}}}\cos\left(-e^{\bar{\mathbf{u}}}\sin\bar{\mathbf{v}}\right)\right|d\bar{\mathbf{v}}\right|.
\end{split}
\end{equation}
When $n=1$, by the relation in \eqref{Eq 77 a1 fu1-a-ijg}, we have from \eqref{Eqcnzgset-7} that
\begin{equation}\label{Eqcnzgset-8}
\begin{split}
\left|\mathbf{I}_{12}\right|\leq \left|\frac{2\zeta}{\iota_1\mathbf{y}_{\mathbf{k}}^{n-1}} e^{-e^{\mathbf{u}+\epsilon(x)}\cos\zeta}\right|.
\end{split}
\end{equation}
When $n\geq 2$, by taking the first derivative of $e^{\bar{\mathbf{u}}}\cos\bar{\mathbf{v}}$ with respect to $\bar{\mathbf{v}}$ we find that
\begin{equation}\label{Eqcnzgset-8bou}
\begin{split}
\frac{d(e^{\bar{\mathbf{u}}}\cos\bar{\mathbf{v}})}{d\bar{\mathbf{v}}}=e^{\bar{\mathbf{u}}}\left(\frac{\bar{\mathbf{u}}_t}{\bar{\mathbf{v}}_t}\cos\bar{\mathbf{v}}-\sin\bar{\mathbf{v}}\right)=e^{\bar{\mathbf{u}}}\left(\frac{\bar{\mathbf{u}}_t}{\bar{\mathbf{v}}_t}-\frac{\sin\bar{\mathbf{v}}}{\cos\bar{\mathbf{v}}}\right)\cos\bar{\mathbf{v}}.
\end{split}
\end{equation}
Note that $\mathbf{u}_y/\mathbf{v}_y<0$ by the estimate in \eqref{Eq 77 a1 fu1aojr-3hilu-3}. Thus, by \eqref{Eqcnzgset-8bou}, we see that, when $\zeta\in[0,\pi/2-\varepsilon]$, $-e^{\bar{\mathbf{u}}}\cos\bar{\mathbf{v}}$ is increasing as $\bar{\mathbf{v}}$ varies from $2\mathbf{k}\pi$ to $2\mathbf{k}\pi+\zeta$; when $\zeta\in[-\pi/2+\varepsilon,0]$, $-e^{\bar{\mathbf{u}}}\cos\bar{\mathbf{v}}$ may be decreasing first and then increases as $\bar{\mathbf{v}}$ varies from $2\mathbf{k}\pi+\zeta$ to $2\mathbf{k}\pi$. In the first case, we always have $-e^{\bar{\mathbf{u}}}\cos\bar{\mathbf{v}}\leq -e^{\mathbf{u}}\cos\zeta$; in the latter case, we have $-e^{\bar{\mathbf{u}}}\cos\bar{\mathbf{v}}\leq \max\{-e^{\mathbf{u}}\cos\zeta,-e^{\mathbf{u}(x,\mathbf{y}_{\mathbf{k}})}\}$. By the relation in \eqref{Eq 77 a1 fu1-a-ijg} we have
\begin{equation*}
\begin{split}
e^{\mathbf{u}(x,\mathbf{y}_{\mathbf{k}})}=e^{\mathbf{u}(x,y)}e^{\mathbf{u}(x,\mathbf{y}_{\mathbf{k}})-\mathbf{u}(x,y)}=e^{\mathbf{u}}e^{-\zeta[1+\epsilon(x)]\frac{\mathbf{u}_y}{\mathbf{v}_y}}.
\end{split}
\end{equation*}
Denote
\begin{equation*}
\begin{split}
\phi_1(x)=\log\left(e^{-\zeta[1+\epsilon(x)]\frac{\mathbf{u}_y}{\mathbf{v}_y}}\frac{1}{\cos\zeta}\right).
\end{split}
\end{equation*}
Note that $\phi_1(x)=O(1)$ and the quantity $O(1)$ is dependent on $\varepsilon$. Denote $\phi_2(x)=\max\{0,\phi_1(x)\}$. Then we have from \eqref{Eqcnzgset-7} that
\begin{equation}\label{Eqcnzgset-8bou-1}
\begin{split}
\left|\mathbf{I}_{12}\right|\leq  \left|\frac{2\zeta}{\iota_1\mathbf{y}_{\mathbf{k}}^{n-1}} e^{-e^{\mathbf{u}+\phi_2(x)}\cos\zeta}\right|
\end{split}
\end{equation}
when $n\geq 2$. Now, by the expression of $\mathbf{u}$ along the ray $\Phi$ in \eqref{Eq 66vfiytdyst}, we see that
\begin{equation}\label{Eq 7gvctdnhli}
\begin{split}
\mathbf{u}=\frac{\cos n\theta}{(\cos\theta)^n}x^n\left[1+\epsilon(x)+O(x^{-n}\ln x)\right].
\end{split}
\end{equation}
Obviously, $\mathbf{v}_y\geq 1$ and hence $\iota_1\mathbf{y}_{\mathbf{k}+1}^{n-1}\geq 1$ along the ray $\Phi$ when $y$ is large. Thus, by \eqref{Eq 7gvctdnhli}, we see from \eqref{Eqcnzgset-8} and \eqref{Eqcnzgset-8bou-1} that $\mathbf{I}_{12}\to 0$ as $x\to\infty$ and satisfies $\mathbf{I}_{12}=\epsilon(\mathbf{y}_{\mathbf{k}})$ in both of the two cases $n=1$ and $n\geq 2$. Then it follows from \eqref{Eqcnzgset}, \eqref{Eqcnzgset-1} and \eqref{Eqcnzgset-6} that
\begin{equation}\label{Eqcnzgset-9}
\begin{split}
\left|\mathbf{H}_1\right|\geq\frac{\sqrt[n]{\sin n\theta}\cos(n-1)\vartheta}{n\sin\theta}(\sin n\vartheta)^{\frac{n-1}{n}}[1+\epsilon(\mathbf{y}_{\mathbf{k}})]\mathbf{y}_{\mathbf{k}}-\epsilon(\mathbf{y}_{\mathbf{k}}).
\end{split}
\end{equation}
For $\mathbf{I}_{21}$, we see by the definition of $\mathbf{J}_k$ in \eqref{Eq ye4evrfnwli-2} that $\mathbf{I}_{21}=\text{Re}(\mathbf{J}_{k})$. Then we may take the real parts on both sides of equation \eqref{Eq ye4evrfnwli-3} and obtain
\begin{equation}\label{Eqcnzgset-5-fu}
\begin{split}
\mathbf{I}_{21}=\text{Re}(\mathbf{J}_{\mathbf{k}})=\text{Re}(\mathbf{G}_{\mathbf{k}})+\text{Re}(d_{\mathbf{k}}-d_{0})+o(1), \quad x\to\infty.
\end{split}
\end{equation}
Similarly as before, we obtain from the estimates in \eqref{Eqcnzgset-4gui} and \eqref{Eq 70-corafeaf4giy} that
\begin{equation}\label{Eqcnzgset-6-1}
\begin{split}
\left|\mathbf{I}_{21}\right|\geq\frac{\sqrt[n]{\sin n\theta}\sin(n-1)\vartheta}{n\sin\theta}(\sin n\vartheta)^{\frac{n-1}{n}}[1+\epsilon(\mathbf{y}_{\mathbf{k}})]\mathbf{y}_{\mathbf{k}}+O(1)
\end{split}
\end{equation}
in both of the two cases $n=1$ and $n \geq 2$. Note that $|\mathbf{I}_{21}|=O(1)$ when $n=1$. Also, by similar arguments as for the estimate in \eqref{Eqcnzgset-8} and \eqref{Eqcnzgset-8bou-1}, we have
\begin{equation}\label{Eqcnzgset-8-1}
\begin{split}
\left|\mathbf{I}_{22}\right|\leq \left| \frac{2\zeta}{\kappa_1\mathbf{y}_{\mathbf{k}}^{n-1}}e^{-e^{\mathbf{u}+\epsilon(x)}\cos\zeta}\right|
\end{split}
\end{equation}
when $n=1$ and
\begin{equation}\label{Eqcnzgset-8-2}
\begin{split}
\left|\mathbf{I}_{22}\right|\leq \left| \frac{2\zeta}{\kappa_1\mathbf{y}_{\mathbf{k}}^{n-1}}e^{-e^{\mathbf{u}+\phi_3(x)}\cos\zeta}\right|
\end{split}
\end{equation}
when $n\geq 2$, for a function $\phi_3(x)$ such that $|\phi_3(x)|=O(1)$. By \eqref{Eq 7gvctdnhli}, we see from \eqref{Eqcnzgset-8-1} and \eqref{Eqcnzgset-8-2} that $\mathbf{I}_{22}\to 0$ as $x\to\infty$ and $\mathbf{I}_{22}=\epsilon(\mathbf{y}_{\mathbf{k}})$ in both of the two cases $n=1$ and $n\geq 2$. Then it follows from \eqref{Eqcnzgset}, \eqref{Eqcnzgset-qari-2} and \eqref{Eqcnzgset-6-1} that
\begin{equation}\label{Eqcnzgset-10}
\begin{split}
\left|\mathbf{H}_2\right|\geq\frac{\sqrt[n]{\sin n\theta}\sin(n-1)\vartheta}{n\sin\theta}(\sin n\vartheta)^{\frac{n-1}{n}}[1+\epsilon(\mathbf{y}_{\mathbf{k}})]\mathbf{y}_{\mathbf{k}}+O(1).
\end{split}
\end{equation}
Recall that $y=x\tan\theta$ along the curve $\Phi$. Then, by the definitions of $\mathbf{H}$ in \eqref{Eq 76fiyt-1} and $\mathbf{H}_1$ and $\mathbf{H}_2$ in \eqref{Eq 76fiyt-2}, in both of the two cases $n=1$ and $n\geq 2$, we have from \eqref{Eq 70-blkgl-1}, \eqref{Eqcnzgset-9} and \eqref{Eqcnzgset-10} that
\begin{equation}\label{Eqcnzgset-11}
\begin{split}
\left|\mathbf{H}\right|\geq\frac{\sqrt[n]{\sin n\theta}}{n\cos\theta}\left[1+\epsilon(x)+O(x^{-n}\ln x)\right]x.
\end{split}
\end{equation}
Therefore, by \eqref{Eq 76fiyt}, \eqref{Eq 76fiyt-1}, \eqref{Eqcnzgset-4gui} and \eqref{Eqcnzgset-11} together with the relation in \eqref{Eq 7gvctdnhli}, we see that, for each integration constant $c$, the solution $f(z)$ in \eqref{Eq 63} satisfies the inequality in \eqref{Eq 3} for $z=x+iy=re^{i\theta}\in \Phi$ such that $\mathbf{v}=2\mathbf{k}\pi+\zeta$ with a large integer $\mathbf{k}$ and a number $\zeta\in[-\pi/2+\varepsilon,\pi/2-\varepsilon]$.

Consider next $\mathbf{v}=2\mathbf{k}\pi+\zeta$ with a number $\zeta\in[\pi/2+\varepsilon,3\pi/2-\varepsilon]$. By looking at the process of estimating $\text{Im}(\mathbf{G}_{k}-\mathbf{G}_{k-1})$ in \eqref{Eq 7giuivftrd} and $\text{Re}(\mathbf{G}_{k}-\mathbf{G}_{k-1})$ in \eqref{Eq 7giuivftrd-1} together with the relations in \eqref{Eq 66vfiytdystgiy-1} and in \eqref{Eqcnzgset-5} and \eqref{Eqcnzgset-5-fu}, we easily find that $\mathbf{I}_{11}$ and $\mathbf{I}_{21}$ in \eqref{Eqcnzgset} both satisfy $\mathbf{I}_{11}=O(\mathbf{y}_{\mathbf{k}}^n)=O(y^n)$ and $\mathbf{I}_{21}=O(\mathbf{y}_{\mathbf{k}}^n)=O(y^n)$. For $\mathbf{I}_{12}$, by taking the imaginary parts on both sides of equation \eqref{Eq 4giubjb-new-5}, we find
\begin{equation*}
\begin{split}
\text{Im}\left(\mathbf{F}_{\mathbf{k}}\right)+\text{Im}\left(\mathbf{J}_{\mathbf{k},\zeta}\right)=\text{Im}\left(\mathbf{G}_{\mathbf{k},\zeta}\right)+\text{Im}\left(\mathbf{F}_{\mathbf{k},\zeta}\right)
\end{split}
\end{equation*}
and, by the definition of $\mathbf{J}_{\mathbf{k},\zeta}$ in \eqref{Eq 4giubjb-new-4}, thus
\begin{equation}\label{Eq 4giubjb-new-14}
\begin{split}
\mathbf{I}_{12}=\text{Im}\left(\mathbf{J}_{\mathbf{k},\zeta}\right)=\text{Im}\left(\mathbf{G}_{\mathbf{k},\zeta}\right)+\text{Im}\left(\mathbf{F}_{\mathbf{k},\zeta}\right)-\text{Im}\left(\mathbf{F}_{\mathbf{k}}\right).
\end{split}
\end{equation}
Note that $\cos\zeta<0$. Then by the estimate for $\mathbf{F}_k$ in \eqref{Eqcnzgset-4gui}, the estimate for $\text{Im}(\mathbf{F}_{\mathbf{k},\zeta})$ in \eqref{Eq 4giubjb-new-8} and the estimate for $\text{Im}(\mathbf{G}_{\mathbf{k},\zeta})$ in \eqref{Eq 4giubjb-new-11}, we have from \eqref{Eq 4giubjb-new-14} that
\begin{equation}\label{Eq 4giubjb-new-15}
\begin{split}
\left|\mathbf{I}_{12}\right|\leq \kappa[1+\epsilon(x)]xe^{-e^{\mathbf{u}}\cos\zeta}+O(1).
\end{split}
\end{equation}
Similarly, for $\mathbf{I}_{22}$, by taking the real parts on both sides of equation \eqref{Eq 4giubjb-new-5} together with the estimate for $\mathbf{F}_k$ in \eqref{Eqcnzgset-4gui}, the estimate for $\text{Re}(\mathbf{F}_{\mathbf{k},\zeta})$ in \eqref{Eq 4giubjb-new-9} and the estimate for $\text{Re}(\mathbf{G}_{\mathbf{k},\zeta})$ in \eqref{Eq 4giubjb-new-12}, we finally have
\begin{equation}\label{Eq 4giubjb-new-16}
\begin{split}
\left|\mathbf{I}_{22}\right|\leq xe^{-e^{\mathbf{u}}\cos\zeta}+O(1).
\end{split}
\end{equation}
Note that $|e^{U(z)}|=e^{e^{\mathbf{u}}\cos\zeta}$ and $\cos\zeta<0$. Then, since $\mathbf{I}_{11}=O(y^n)$ and $\mathbf{I}_{21}=O(y^n)$, by \eqref{Eq 76fiyt}, \eqref{Eq 76fiyt-1}, \eqref{Eqcnzgset}, \eqref{Eqcnzgset-1}, \eqref{Eqcnzgset-qari-2} and the inequalities in \eqref{Eq 4giubjb-new-15} and \eqref{Eq 4giubjb-new-16}, we see that, for each integration constant $c$, the solution $f(z)$ in \eqref{Eq 63} satisfies $|f(z)|=O(x)$ for $z=x+iy\in\Phi$ such that $\mathbf{v}=2\mathbf{k}\pi+\zeta$ with a large integer $\mathbf{k}$ and a number $\zeta\in[\pi/2+\varepsilon,3\pi/2-\varepsilon]$.

On the basis of the previous estimates for $|f(z)|$ of $f(z)$ in \eqref{Eq 63} along the ray $\Phi$, we define a sequence $\{r_k\}$ in the way that $r_k\sin\theta=\tilde{x}_k$ and $r_k\cos\theta=\tilde{y}_k$ and $\mathbf{v}(\tilde{x}_{2l},\tilde{y}_{2l})=2(2l+\mathbf{k})\pi-\pi/2+\varepsilon$ if $k=2l$ and $\mathbf{v}(\tilde{x}_{2l+1},\tilde{y}_{2l+1})=2(2l+\mathbf{k})\pi+\pi/2-\varepsilon$. Then the inequality in \eqref{Eq 3} holds for all $z=x+iy=re^{i\theta}$ such that $r\in E=\cup_{l=0}^{\infty}[r_{2l},r_{2l+1}]$. Moreover, by the estimate for $\mathbf{v}$ in \eqref{Eq 66vfiytdystbk}, we have
\begin{equation}\label{Eq 66gragbbsgfaf}
\begin{split}
\mathbf{v}(\tilde{x}_{k},\tilde{y}_{k})&=\frac{\sin n\theta}{(\sin\theta)^n}\tilde{y}_{k}^n[1+\epsilon(\tilde{y}_{k})].
\end{split}
\end{equation}
Obviously, $\tilde{y}_{k}\to\infty$ as $k\to\infty$. Denote $\Delta x=\tilde{x}_{2l+1}-\tilde{x}_{2l}$ and $\Delta y=\tilde{y}_{2l+1}-\tilde{y}_{2l}$ for simplicity. Since $\mathbf{v}(\tilde{x}_{2l+1},\tilde{y}_{2l+1})-\mathbf{v}(\tilde{x}_{2l},\tilde{y}_{2l})=\pi-2\varepsilon$, we deduce from \eqref{Eq 66gragbbsgfaf} that
\begin{equation*}
\begin{split}
\tilde{y}_{2l+1}[1+\epsilon(\tilde{y}_{2l+1})]-\tilde{y}_{2l}[1+\epsilon(\tilde{y}_{2l})]=O(1).
\end{split}
\end{equation*}
Thus $\Delta y=O(1)$ and also $\Delta x=O(1)$. By Lagrange's mean value theorem together with the Cauchy--Riemann equations, there are two constants $\delta_3,\delta_4\in(0,1)$ such that
\begin{equation*}
\begin{split}
\pi-2\varepsilon&=\mathbf{v}(\tilde{x}_{2l+1},\tilde{y}_{2l+1})-\mathbf{v}(\tilde{x}_{2l},\tilde{y}_{2l})\\
&=\left[\mathbf{v}_x(\tilde{x}_{2l}+\delta_3\Delta x,\tilde{y}_{2l}+\delta_4\Delta y)\frac{1}{\tan\theta}+\mathbf{v}_y(\tilde{x}_{2l}+\delta_3\Delta x,\tilde{y}_{2l}+\delta_4\Delta y)\right]\Delta y\\
&=\mathbf{v}_y(\tilde{x}_{2l}+\delta_3\Delta x,\tilde{y}_{2l}+\delta_4\Delta y)\left[1-\frac{1}{\tan\theta}\frac{\mathbf{u}_y(\tilde{x}_{2l}+\delta_3\Delta x,\tilde{y}_{2l}+\delta_4\Delta y)}{\mathbf{v}_y(\tilde{x}_{2l}+\delta_3\Delta x,\tilde{y}_{2l}+\delta_4\Delta y)}\right]\Delta y.
\end{split}
\end{equation*}
Then, by the estimate for $\mathbf{v}_y$ in \eqref{Eq 77 a1 fu1aojr-3hilu-1-fu} and the two estimates for $\mathbf{u}_y/\mathbf{v}_y$ in \eqref{Eq 77 a1 fu1aojr-3hilu-2} and \eqref{Eq 77 a1 fu1aojr-3hilu-3} together with the equation in \eqref{Eq 66gragbbsgfaf}, we obtain
\begin{equation*}
\begin{split}
\Delta y=\frac{(\pi-2\varepsilon)(\sin n\theta)^{1-1/n}\tan\theta}{n[\tan\theta+\tan(n-1)\theta]\cos (n-1)\theta}\cdot\frac{1+\epsilon(\tilde{y}_{2l})}{[2(2l+\mathbf{k})\pi-\pi/2+\varepsilon]^{1-1/n}}
\end{split}
\end{equation*}
in both of the two cases $n=1$ and $n\geq 2$. Then the set $E=\cup_{l=0}^{\infty}[r_{2l},r_{2l+1}]$ has infinite logarithmic measure, i.e., $\int_{E}dr/r=\infty$. This completes the proof.

\section{A further lower bound for $|f(z)|$ of $f(z)$ in \eqref{Eq 4}}\label{Lower bound}

In this section, we continue to estimate the modulus $|f(z)|$ for the solution $f(z)$ in \eqref{Eq 4} along the ray $\Phi$ on which $z=re^{i\theta}$, $r\in[r_0,\infty)$. As a complement of the results in the introduction, we shall provide a lower bound for $|f(z)|$ for $z=x+iy\in \Phi$ such that $y=2\mathbf{k}\pi+\pi+\eta$ with the integer $\mathbf{k}$ and a number $\eta\in[-\pi/2+\varepsilon,\pi/2-\varepsilon]$. Recall the expression of $f(z)$ along the curve $L$ in \eqref{Eq 70-cbuy-0}. Since $|e^{e^z}|=e^{-e^x\cos\eta}$ and since $\int_{x_0}^{x}e^{-e^t}dt=d_0+o(1)$ for some constant $d_0$ by \eqref{Eq 4nalr-1}, it suffices to provide a lower bound for $|\mathbf{H}|$ of $\mathbf{H}$ in \eqref{Eq 70-cbuy}. We shall prove the following

\begin{theorem}\label{maintheorem2}
For $z=x+iy\in\Phi$ such that $y=2\mathbf{k}\pi+\pi+\eta$ with a large integer $\mathbf{k}>0$ and a number $\eta\in[-\pi/2+\varepsilon,\pi/2-\varepsilon]$, there is a constant $\xi=\xi(\varepsilon)>0$ dependent on $\varepsilon$ such that the function $\mathbf{H}$ in \eqref{Eq 70-cbuy} satisfies
\begin{equation}\label{Eq 3gouu}
\left|\mathbf{H}\right|>\xi \frac{e^{e^{x}\cos\eta}}{e^{x}\cos\eta}.
\end{equation}
\end{theorem}

\begin{proof}

We use the notations in the introduction. Now, under the assumptions of Theorem~\ref{maintheorem2}, we may write $\mathbf{H}$ in \eqref{Eq 70-cbuy} as
\begin{equation}\label{Eq 70-cbuy-a-1}
\begin{split}
\mathbf{H}=\int_{0}^{y}e^{-\mathbf{r}\cos t}e^{-i\mathbf{r}\sin t}dt=\mathbf{H}_1+i\mathbf{H}_2,
\end{split}
\end{equation}
where $\mathbf{r}=e^x$ and
\begin{equation}\label{Eq 70-cbuy-a-2}
\begin{split}
\mathbf{H}_1=\int_{0}^{2\mathbf{k}\pi+\pi+\eta}e^{-\mathbf{r}\cos t}\cos(-\mathbf{r}\sin t)dt,\\
\mathbf{H}_2=\int_{0}^{2\mathbf{k}\pi+\pi+\eta}e^{-\mathbf{r}\cos t}\sin(-\mathbf{r}\sin t)dt.
\end{split}
\end{equation}
Then, by the two identities in \eqref{Eq 70-cor4} and the discussions in the introduction, it is easy to deduce that
\begin{equation}\label{Eq 70-cbuy-a-3}
\begin{split}
\mathbf{H}_1&=2(\mathbf{k}+1)\pi-\int_{0}^{\pi/2}e^{-\mathbf{r}\cos t}\cos(\mathbf{r}\sin t)dt-\int_{\eta}^{\pi/2}e^{\mathbf{r}\cos t}\cos(\mathbf{r}\sin t)dt,\\
\mathbf{H}_2&=-\int_{0}^{\pi/2}e^{-\mathbf{r}\cos t}\sin(\mathbf{r}\sin t)dt-\int_{\eta}^{\pi/2}e^{\mathbf{r}\cos t}\sin(\mathbf{r}\sin t)dt
\end{split}
\end{equation}
when $\eta\geq0$ and
\begin{equation}\label{Eq 70-cbuy-a-4}
\begin{split}
\mathbf{H}_1&=2\mathbf{k}\pi+\int_{0}^{\pi/2}e^{-\mathbf{r}\cos t}\cos(\mathbf{r}\sin t)dt+\int_{-\eta}^{\pi/2}e^{\mathbf{r}\cos t}\cos(\mathbf{r}\sin t)dt,\\
\mathbf{H}_2&=-\int_{0}^{\pi/2}e^{-\mathbf{r}\cos t}\sin(\mathbf{r}\sin t)dt-\int_{-\eta}^{\pi/2}e^{\mathbf{r}\cos t}\sin(\mathbf{r}\sin t)dt
\end{split}
\end{equation}
when $\eta<0$. Note that
$|\int_{0}^{\pi/2}e^{-\mathbf{r}\cos t}\cos(\mathbf{r}\sin t)dt|<\pi/2$ and
$|\int_{0}^{\pi/2}e^{-\mathbf{r}\cos t}\sin(\mathbf{r}\sin t)dt|<\pi/2$. Thus below we shall estimate the two quantities
\begin{equation}\label{Eq 70-cbuy-a-5}
\begin{split}
\mathbf{H}_{3}&=\int_{|\eta|}^{\pi/2}e^{\mathbf{r}\cos t}\cos(\mathbf{r}\sin t)dt,\\
\mathbf{H}_{4}&=\int_{|\eta|}^{\pi/2}e^{\mathbf{r}\cos t}\sin(\mathbf{r}\sin t)dt.
\end{split}
\end{equation}
Also note that, for any constant $\delta>0$ such that $\pi/2-\delta\in(|\eta|,\pi/2]$,
\begin{equation}\label{E qtwxaa}
\begin{split}
\left|\int_{\pi/2-\delta}^{\pi/2}e^{\mathbf{r}\cos t}\cos(\mathbf{r}\sin t)dt\right|&\leq \delta e^{\mathbf{r}\sin \delta}=\delta e^{\mathbf{r}(\sin \delta-\cos\eta)}e^{\mathbf{r}\cos \eta},\\
\left|\int_{\pi/2-\delta}^{\pi/2}e^{\mathbf{r}\cos t}\sin(\mathbf{r}\sin t)dt\right|&\leq \delta e^{\mathbf{r}\sin \delta}=\delta e^{\mathbf{r}(\sin \delta-\cos\eta)}e^{\mathbf{r}\cos \eta}.
\end{split}
\end{equation}
Thus it suffices to look at the two quantities
\begin{equation}\label{j;aljf;a}
\begin{split}
\mathbf{I}_1&=\int_{|\eta|}^{\pi/2-\delta}e^{\mathbf{r}\cos t}\cos(\mathbf{r}\sin t)dt,\\
\mathbf{I}_2&=\int_{|\eta|}^{\pi/2-\delta}e^{\mathbf{r}\cos t}\sin(\mathbf{r}\sin t)dt,
\end{split}
\end{equation}
where $\eta$ is a number such that
$\mathbf{r}\sin|\eta|=2k_\mathbf{m}\pi+\lambda$, $\lambda\in \left[0,2\pi\right]$
for an integer $k_{\mathbf{m}}\geq 0$.

We may always assume that the constant $\delta$ in \eqref{E qtwxaa} satisfies $\mathbf{r}\sin(\pi/2-\delta)=2k_{\mathbf{n}}\pi+2\pi+\lambda$ for some integer $k_{\mathbf{n}}$ and also that $\mathbf{r}^2\sin^2(\pi/2-\delta)\leq \mathbf{r}^2-1$. Under this assumption, we look at $\mathbf{I}_1$ in \eqref{j;aljf;a}. Set $\mathbf{r}\sin t=w$. Then $\mathbf{I}_1$ in \eqref{j;aljf;a} becomes
\begin{equation}\label{Eq 20}
\begin{split}
\mathbf{I}_1=\int_{\mathbf{r}\sin|\eta|}^{\mathbf{r}\sin(\pi/2-\delta)}\chi(w)\cos wdw,
\end{split}
\end{equation}
where
\begin{equation}\label{Eq 24a-fjhi}
\begin{split}
\chi(w)=\chi(w,\mathbf{r})=\frac{e^{\sqrt{\mathbf{r}^2-w^2}}}{\sqrt{\mathbf{r}^2-w^2}}.
\end{split}
\end{equation}
By the periodicity of $\sin w$, we easily deduce from equation \eqref{Eq 20} that
\begin{equation}\label{Eq 24hilu}
\begin{split}
\mathbf{I}_1=\int_{\lambda}^{\lambda+\pi}\left[\sum_{k=2k_{\mathbf{m}}}^{2k_{\mathbf{n}}+1}(-1)^{k}\chi(s+k\pi)\right]\cos sds.
\end{split}
\end{equation}
By taking the first derivative of $\chi(w)$ with respect to $w$, we find
\begin{equation*}
\begin{split}
\chi'(w)=\frac{-we^{\sqrt{\mathbf{r}^2-w^2}}}{\sqrt{\mathbf{r}^2-w^2}}\cdot\frac{\sqrt{\mathbf{r}^2-w^2}-1}{\mathbf{r}^2-w^2}.
\end{split}
\end{equation*}
We see that $\chi'(w)<0$ when $\mathbf{r}^2-\mathbf{r}^2\sin^2(\pi/2-\delta)\geq 1$. Thus the difference $\chi(s+k\pi)-\chi(s+(k+1)\pi)$ is always positive. The function $\mathbf{I}_2$ can be dealt with in a similar way. Below we fix a small constant $\tau>0$ and consider the two cases $|\eta|\in [\tau,\pi/2-\varepsilon]$ and $|\eta|\in[0,\tau]$ respectively.

\subsection{Part~I}

If $|\eta|\in[\tau,\pi/2-\varepsilon]$, we will estimate $\mathbf{I}_1$ when $\lambda\in[\pi/2,\pi]$ or $\lambda\in[3\pi/2,2\pi]$ and estimate $\mathbf{I}_2$ when $\lambda\in [0,\pi/2]$ or $\lambda\in [\pi,3\pi/2]$.

Look at $\mathbf{I}_1$ in \eqref{j;aljf;a} for the case when $\lambda\in[\pi/2,\pi]$. Denote $s_k=w/\mathbf{r}$ and $\tilde{s}_k=k\pi/\mathbf{r}$ for simplicity. Then, for the value $w=s+k\pi$ and each of the even integers $k=2k_{\mathbf{m}},\cdots,2k_{\mathbf{n}}$, by \eqref{Eq 24a-fjhi} we have
\begin{equation}\label{Eqohulh-fj-1}
\begin{split}
\frac{\chi(w+\pi)}{\chi(w)}&=\exp\left({\frac{-2\pi\tilde{s}_k-(2s+\pi)/\mathbf{r}}{\sqrt{1-s_{k+1}^2}+\sqrt{1-s_k^2}}}\right)
\sqrt{1+\frac{s_{k+1}^2-s_{k}^2}{1-s_{k+1}^2}}.
\end{split}
\end{equation}
Denote $\alpha_{k}=\tilde{s}_{k}/\sqrt{1-\tilde{s}_{k}^2}$ for simplicity. Since $\sin\tau\leq \tilde{s}_{k}\leq \cos\varepsilon$, we see that $\tan\tau\leq \alpha_{k}\leq 1/\tan\varepsilon$ uniformly for all the even integers $k=2k_{\mathbf{m}},\cdots,2k_{\mathbf{n}}$. Note that $s_{k+1}^2-s_k^2=[2s+(2k+1)\pi]\pi/\mathbf{r}^2=O(\mathbf{r}^{-1})$. Similarly, $s_{k}^2-\tilde{s}_k^2=O(\mathbf{r}^{-1})$ and $s_{k+1}^2-\tilde{s}_k^2=O(\mathbf{r}^{-1})$. Then, using the relations $e^t\sim 1+t$ as $t\to 0$ and $\sqrt{1+t}\sim 1+t/2$ as $t\to 0$, we find
\begin{equation}\label{Eqohulh-fj-2}
\begin{split}
\frac{\chi(w+\pi)}{\chi(w)}=e^{-\pi\alpha_{k}}\left[1+O\left(\frac{1}{\mathbf{r}}\right)\right]
\end{split}
\end{equation}
and also
\begin{equation}\label{Eqohulh-fj-2ho}
\begin{split}
\frac{\chi((k+1)\pi)}{\chi(k\pi)}=e^{-\pi\alpha_{k}}\left[1+O\left(\frac{1}{\mathbf{r}}\right)\right].
\end{split}
\end{equation}
Note that the quantities $O(1)$ in \eqref{Eqohulh-fj-2} and \eqref{Eqohulh-fj-2ho} are dependent on $\varepsilon$. Since $(s+k\pi)^2-(k\pi)^2=(s+2k\pi)s$, $\chi(w)/\chi(k\pi)$ has similar estimate, namely
\begin{equation}\label{Eqohulh-fj-3}
\begin{split}
\frac{\chi(w)}{\chi(k\pi)}=e^{-\alpha_{k}s}\left[1+O\left(\frac{1}{\mathbf{r}}\right)\right].
\end{split}
\end{equation}
By \eqref{Eqohulh-fj-2} and \eqref{Eqohulh-fj-2ho} it follows that
\begin{equation}\label{Eqohulh-fj-4}
\begin{split}
\frac{\chi(w)-\chi(w+\pi)}{\chi(k\pi)-\chi((k+1)\pi)}=\frac{\chi(w)}{\chi(k\pi)}\frac{1-\frac{\chi(w+\pi)}{\chi(w)}}{1-\frac{\chi((k+1)\pi)}{\chi(k\pi)}}=\frac{\chi(w)}{\chi(k\pi)}\left[1+O\left(\frac{1}{\mathbf{r}}\right)\right].
\end{split}
\end{equation}
Now, for the value $w=s+k\pi$ and each of the even integers $k=2k_{\mathbf{m}},\cdots,2k_{\mathbf{n}}$, it follows from \eqref{Eqohulh-fj-3} and \eqref{Eqohulh-fj-4} that
\begin{equation}\label{Eq giyji-fj}
\begin{split}
\int_{\lambda}^{\lambda+\pi}\left[\chi(w)-\chi(w+\pi)\right]\cos sds=\left[\chi(k\pi)-\chi((k+1)\pi)\right]\Lambda_{k,1},
\end{split}
\end{equation}
where
\begin{equation*}
\begin{split}
\Lambda_{k,1}=\int_{\lambda}^{\pi+\lambda}\left[1+O\left(\frac{1}{\mathbf{r}}\right)\right]e^{-\alpha_{k}s}\cos sds.
\end{split}
\end{equation*}
For any $\lambda\in[\pi/2,\pi]$, elementary calculation and estimation yield
\begin{equation}\label{Eq giyji-fjbku-j2}
\begin{split}
\Lambda_{k,1}=\frac{e^{-\alpha_{k}(\lambda+\pi)}+e^{-\alpha_{k}\lambda}}{1+\alpha_{k}^2}\left(\alpha_{k}\cos\lambda-\sin\lambda\right)+O\left(\frac{1}{\mathbf{r}}\right).
\end{split}
\end{equation}
Since $\sin \lambda>0$ and $\cos \lambda<0$ for any $\lambda\in[\pi/2,\pi]$ and $\tan\tau\leq \alpha_{k}\leq 1/\tan\varepsilon$ uniformly for all the even integers $k=2k_{\mathbf{m}},\cdots,2k_{\mathbf{n}}$, we see that $\Lambda_{k,1}$ is always negative and there is a constant $\xi_1=\xi_1(\tau,\varepsilon)$ dependent on $\tau$ and $\varepsilon$ such that $|\Lambda_{k,1}|\geq \xi_1$ uniformly for all the even integers $k=2k_{\mathbf{m}},\cdots,2k_{\mathbf{n}}$.

Recall that $\chi(s+k\pi)-\chi(s+(k+1)\pi)$ is always positive. By \eqref{Eq 24hilu}, we summarize the integrals in \eqref{Eq giyji-fj} for all the even integers $k=2k_{\mathbf{m}}\cdots,2k_{\mathbf{n}}$ and obtain that $\mathbf{I}_1$ is negative and
\begin{equation}\label{Eq 24a-fj}
\begin{split}
|\mathbf{I}_1|>\xi_1\sum_{k=2k_{\mathbf{m}}}^{2k_{\mathbf{n}}+1}(-1)^k\chi(k\pi)>\xi_1\frac{e^{\mathbf{r}\cos\eta}}{\mathbf{r}\cos\eta}\frac{\chi(2k_{\mathbf{m}}\pi)}{\chi(\mathbf{r}\sin|\eta|)}\left[1-\frac{\chi((2k_{\mathbf{m}}+1)\pi)}{\chi(2k_{\mathbf{m}}\pi)}\right].
\end{split}
\end{equation}
Together with the estimate in \eqref{Eqohulh-fj-2ho} and a similar estimate for $\chi(2k_{\mathbf{m}}\pi)/\chi(\mathbf{r}\sin|\eta|)$ as the one in \eqref{Eqohulh-fj-2}, we see from \eqref{Eq 24a-fj} that there is a constant $\xi_2=\xi_2(\tau,\varepsilon)>0$ dependent on $\tau$ and $\varepsilon$ such that
\begin{equation}\label{Eq 24a-fjhggouky}
\begin{split}
|\mathbf{I}_1|>\xi_1\xi_2\frac{e^{\mathbf{r}\cos\eta}}{\mathbf{r}\cos\eta}.
\end{split}
\end{equation}
By the estimate in \eqref{E qtwxaa} we have $|\mathbf{H}_{3}|\geq (1-\varepsilon)|\mathbf{I}_1|$. Thus the inequality in \eqref{Eq 3gouu} follows by the relations $|\mathbf{H}|\geq |\mathbf{H}_1|\geq (1-\varepsilon)|\mathbf{H}_{3}|$.

When $\lambda\in[3\pi/2,2\pi]$, we also obtain the estimate for $\Lambda_{k,1}$ in \eqref{Eq giyji-fjbku-j2}, which is now positive. Thus, we also have similar inequalities as those in \eqref{Eq 24a-fj} and \eqref{Eq 24a-fjhggouky}.

When $\lambda\in[0,\pi/2]$ or $\lambda\in[\pi,3\pi/2]$, we estimate $\mathbf{I}_2$ in \eqref{j;aljf;a} and derive a similar equation as the one in \eqref{Eq giyji-fj} with $\cos s$ there being replaced by $\sin s$ and $\Lambda_{k,1}$ there by
\begin{equation}\label{Eqhlonpre}
\begin{split}
\Lambda_{k,2}=\frac{e^{-\alpha_{k}(\lambda+\pi)}+e^{-\alpha_{k}\lambda}}{1+\alpha_{k}^2}\left(\alpha_{k}\sin\lambda+\cos\lambda\right)+O\left(\frac{1}{\mathbf{r}}\right).
\end{split}
\end{equation}
When $\lambda\in[0,\pi/2]$ or $\lambda\in[\pi,3\pi/2]$, $\Lambda_{k,2}$ is always negative or positive. The estimate in \eqref{Eqhlonpre} implies similar inequalities for $\mathbf{I}_2$ as those in \eqref{Eq 24a-fj} and \eqref{Eq 24a-fjhggouky}. By the estimate in \eqref{E qtwxaa} we have $|\mathbf{H}_{4}|\geq (1-\varepsilon)|\mathbf{I}_2|$. Thus the inequality in \eqref{Eq 3gouu} follows by the relations $|\mathbf{H}|\geq |\mathbf{H}_2|\geq (1-\varepsilon)|\mathbf{H}_{4}|$.

\subsection{Part~II}

If $|\eta|\in[0,\tau]$, we will estimate $\mathbf{I}_1$ in \eqref{j;aljf;a} when $\lambda\in[\pi/4,3\pi/4]$ or $\lambda\in[5\pi/4,7\pi/4]$ and estimate $\mathbf{I}_2$ in \eqref{j;aljf;a} when $\lambda\in[3\pi/4,5\pi/4]$ or $\lambda\in[7\pi/4,2\pi]\cup[0,\pi/4]$. We choose $\delta$ so that $w/\mathbf{r}\leq \sin 2\tau$ for all the even integers $k=2k_{\mathbf{m}},\cdots,2k_{\mathbf{n}}$.

Look at $\mathbf{I}_1$ in \eqref{j;aljf;a} for the case when $\lambda\in[\pi/4,3\pi/4]$. Now, using the relations $e^t\sim 1+t$ as $t\to 0$ and $\sqrt{1+t}\sim 1+t/2$ as $t\to 0$, we find that the quotient in \eqref{Eqohulh-fj-1}, for the value $w=s+k\pi$, takes the form
\begin{equation*}
\begin{split}
\frac{\chi(w+\pi)}{\chi(w)}&=\exp\left(-{\frac{[2s+(2k+1)\pi]\pi}{2\mathbf{r}\sqrt{1-\tilde{s}_k^2}}}\right)\left[1+O\left(\frac{k+1}{\mathbf{r}^2}\right)\right].
\end{split}
\end{equation*}
Denote $\beta_k=[2s+(2k+1)\pi]\pi/2\mathbf{r}\sqrt{1-\tilde{s}_k^2}$ and $\gamma_k=[2s+(2k+1)\pi]\pi/2\mathbf{r}$ and also $\tilde{\beta}_k=(2k+1)\pi^2/2\mathbf{r}\sqrt{1-\tilde{s}_k^2}$ and $\tilde{\gamma}_k=(2k+1)\pi^2/2\mathbf{r}$ for simplicity. Recall that $\tilde{s}_k=k\pi/\mathbf{r}$. If $\tau$ is small, then using the relations $e^t\sim 1+t$ as $t\to 0$, we find
\begin{equation}\label{Eqohulh-fj-gde-1}
\begin{split}
\frac{\chi(w+\pi)}{\chi(w)}=\left[1-\beta_k+O(\beta_k^2)\right]\left[1+O\left(\frac{k+1}{\mathbf{r}^2}\right)\right]=1-\beta_k\left[1+O\left(\frac{k+1}{\mathbf{r}}\right)\right]
\end{split}
\end{equation}
and also
\begin{equation}\label{Eqohulh-fj-gde-2}
\begin{split}
\frac{\chi((k+1)\pi)}{\chi(k\pi)}=\left[1-\tilde{\beta}_k+O(\tilde{\beta}_k^2)\right]\left[1+O\left(\frac{k+1}{\mathbf{r}^2}\right)\right]=1-\tilde{\beta}_k\left[1+O\left(\frac{k+1}{\mathbf{r}}\right)\right].
\end{split}
\end{equation}
Since $(s+k\pi)^2-(k\pi)^2=(s+2k\pi)s$, $\chi(w)/\chi(k\pi)$ has similar estimate, namely
\begin{equation}\label{Eqohulh-fj-gde-3-a}
\begin{split}
\frac{\chi(w)}{\chi(k\pi)}=1+O\left(\frac{k+1}{\mathbf{r}}\right).
\end{split}
\end{equation}
By \eqref{Eqohulh-fj-gde-1} and \eqref{Eqohulh-fj-gde-2} together with the relation $\sqrt{1-t}\sim 1-t/2$ as $t\to 0$, it follows that
\begin{equation}\label{Eqohulh-fj-gde-3-b}
\begin{split}
\frac{\chi(w)-\chi(w+\pi)}{\chi(k\pi)-\chi((k+1)\pi)}=\frac{\chi(w)}{\chi(k\pi)}\frac{1-\frac{\chi(w+\pi)}{\chi(w)}}{1-\frac{\chi((k+1)\pi)}{\chi(k\pi)}}=\frac{\chi(w)}{\chi(k\pi)}\frac{\gamma_k}{\tilde{\gamma}_k}\left[1+O\left(\frac{k+1}{\mathbf{r}}\right)\right].
\end{split}
\end{equation}
Now, for the value $w=s+k\pi$ and each of the even integers $k=2k_{\mathbf{m}},\cdots,2k_{\mathbf{n}}$, it follows from \eqref{Eqohulh-fj-gde-3-a} and \eqref{Eqohulh-fj-gde-3-b} that
\begin{equation}\label{Eqohulh-fj-gde-3}
\begin{split}
\int_{\lambda}^{\lambda+\pi}\left[\chi(w)-\chi(w+\pi)\right]\cos sds=\left[\chi(k\pi)-\chi((k+1)\pi)\right]\Lambda_{k,3},
\end{split}
\end{equation}
where
\begin{equation*}
\begin{split}
\Lambda_{k,3}=\int_{\lambda}^{\lambda+\pi}\left[1+O\left(\frac{k+1}{\mathbf{r}}\right)\right]\frac{\gamma_k}{\tilde{\gamma}_k}\cos sds.
\end{split}
\end{equation*}
For any $\lambda\in[\pi/4,3\pi/4]$, elementary calculation and estimation yield
\begin{equation}\label{Eqohulh-fj-gde-7}
\begin{split}
\Lambda_{k,3}=-2\sin\lambda-\frac{2}{(2k+1)\pi}\left[(2\lambda+\pi)\sin\lambda+2\cos\lambda\right]+O\left(\frac{k+1}{\mathbf{r}}\right).
\end{split}
\end{equation}
When $\lambda\in[\pi/4,3\pi/4]$ and $\tau$ is small, we see that $\Lambda_{k,3}$ is always negative and there is a constant $\xi_3=\xi_3(\varepsilon)$ dependent on $\varepsilon$ such that $|\Lambda_{k,3}|\geq \xi_3$ uniformly for all the even integers $k=2k_{\mathbf{m}},\cdots,2k_{\mathbf{n}}$.

Recall that $\chi(s+k\pi)-\chi(s+(k+1)\pi)$ is always positive. By \eqref{Eq 24hilu}, we summarize the integrals in \eqref{Eqohulh-fj-gde-3} for all the even integers $k=2k_{\mathbf{m}},\cdots,2k_{\mathbf{n}}$ and obtain that $\mathbf{I}_1$ is negative and
\begin{equation}\label{Eqohulh-fj-gde-8}
\begin{split}
|\mathbf{I}_1|>\xi_3\sum_{k=2k_{\mathbf{m}}}^{2k_{\mathbf{n}}+1}(-1)^k\chi(k\pi)>\xi_3\frac{e^{\mathbf{r}\cos\eta}}{\mathbf{r}\cos\eta}\frac{\chi(2k_{\mathbf{n}}\pi)}{\chi(\mathbf{r}\sin|\eta|)}\left[1-\frac{\chi((2k_{\mathbf{n}}+1)\pi)}{\chi(2k_{\mathbf{n}}\pi)}\right].
\end{split}
\end{equation}
Together with the estimate in \eqref{Eqohulh-fj-gde-2} and a similar estimate for $\chi(2k_{\mathbf{m}}\pi)/\chi(\mathbf{r}\sin|\eta|)$ as the one in \eqref{Eqohulh-fj-gde-1}, we see from \eqref{Eqohulh-fj-gde-8} that there is a constant $\xi_4=\xi_4(\tau)>0$ dependent on $\tau$ such that
\begin{equation}\label{Eqohulh-fj-gde-8-9}
\begin{split}
|\mathbf{I}_1|>\xi_3\xi_4\frac{e^{\mathbf{r}\cos\eta}}{\mathbf{r}\cos\eta}.
\end{split}
\end{equation}
Since $\sin \delta-\cos\eta\leq -\sin^2\tau$ when $\tau$ is small, by the estimate in \eqref{E qtwxaa} we have $|\mathbf{H}_{3}|\geq (1-\varepsilon)|\mathbf{I}_1|$. Thus the inequality in \eqref{Eq 3gouu} follows by the relations $|\mathbf{H}|\geq |\mathbf{H}_1|\geq (1-\varepsilon)|\mathbf{H}_{3}|$.

When $\lambda\in[5\pi/4,7\pi/4]$, we also obtain the estimate for $\Lambda_{k,3}$ in \eqref{Eqohulh-fj-gde-7}, which is now positive. Thus, we also have similar inequalities as those in \eqref{Eqohulh-fj-gde-8} and \eqref{Eqohulh-fj-gde-8-9}.

When $\lambda\in[3\pi/4,5\pi/4]$ or $\lambda\in[7\pi/4,2\pi]\cup[0,\pi/4]$, we estimate $\mathbf{I}_2$ in \eqref{j;aljf;a} and derive a similar equation as the one in \eqref{Eqohulh-fj-gde-3} but with $\cos s$ there being replaced by $\sin s$ and $\Lambda_{k,3}$ there by
\begin{equation}\label{Eqohulh-fj-gde-5gik}
\begin{split}
\Lambda_{k,4}=2\cos\lambda+\frac{2}{(2k+1)\pi}\left[(2\lambda+\pi)\cos\lambda-2\sin\lambda\right]+O\left(\frac{k+1}{\mathbf{r}}\right).
\end{split}
\end{equation}
When $\lambda\in[3\pi/4,5\pi/4]$ or $\lambda\in[7\pi/4,2\pi]\cup[0,\pi/4]$, $\Lambda_{k,4}$ is always negative or positive. The estimate in \eqref{Eqohulh-fj-gde-5gik} implies similar inequalities for $\mathbf{I}_2$ as those in \eqref{Eqohulh-fj-gde-8} and \eqref{Eqohulh-fj-gde-8-9}. By the estimate in \eqref{E qtwxaa} we have $|\mathbf{H}_{4}|\geq (1-\varepsilon)|\mathbf{I}_2|$. Thus the inequality in \eqref{Eq 3gouu} follows by the relations $|\mathbf{H}|\geq |\mathbf{H}_2|\geq (1-\varepsilon)|\mathbf{H}_{4}|$. We omit those details. This completes the proof.

\end{proof}

At the end of this section, we look at the solution in \eqref{Eq 4} by choosing a different path of integration. We may connect the two points $z_0=x_0$ and $z_1=x_0+i(2\mathbf{k}\pi+\pi)$ by a vertical line segment $\mathcal{V}_1$ and the two points $z_1=x_0+i(2\mathbf{k}\pi+\pi)$ and $z_2=x+i(2\mathbf{k}\pi+\pi)$ by a horizontal line segment $\mathcal{L}_1$ and then connect the two points $z_2=x+i(2\mathbf{k}\pi+\pi)$ and $z=x+iy$ by a vertical line segment $\mathcal{V}_2$. Let $\mathcal{L}$ be the joint of $\mathcal{V}_1$, $\mathcal{L}_1$ and $\mathcal{V}_2$. Then, by the Cauchy integral theorem, along the curve $\mathcal{L}$, we have
\begin{equation}\label{Eq 76fiytarjoe-fu}
\begin{split}
f(z)=e^{e^z}\left[c+i\int_{0}^{2\mathbf{k}\pi+\pi}e^{-e^{x_0}e^{it}}dt+\int_{x_0}^{x}e^{e^{s}}ds+i(\mathbf{H}_5+i\mathbf{H}_6)\right],
\end{split}
\end{equation}
where
\begin{equation*}
\begin{split}
\mathbf{H}_5&=\int_{2\mathbf{k}\pi+\pi}^{2\mathbf{k}\pi+\pi+\eta}e^{-\mathbf{r}\cos t}\cos\left(-\mathbf{r}\sin t\right)dt=\int_{0}^{\eta}e^{\mathbf{r}\cos t}\cos\left(\mathbf{r}\sin t\right)dt,\\
\mathbf{H}_6&=\int_{2\mathbf{k}\pi+\pi}^{2\mathbf{k}\pi+\pi+\eta}e^{-\mathbf{r}\cos t}\sin\left(-\mathbf{r}\sin t\right)dt=\int_{0}^{\eta}e^{\mathbf{r}\cos t}\sin\left(\mathbf{r}\sin t\right)dt.
\end{split}
\end{equation*}
We may also prove Theorem~\ref{maintheorem1} and Theorem~\ref{maintheorem2} for the simplest case of $h(z)$ by estimating the terms in \eqref{Eq 76fiytarjoe-fu}. Denote $\mathbf{r}_0=e^{x_0}$ for simplicity. By comparing the expression in \eqref{Eq 76fiytarjoe-fu} with the one in \eqref{Eq 70-cbuy-0} together with the two identities in \eqref{Eq 70-cor4}, we find
\begin{equation}\label{Eq 5}
\begin{split}
\int_{0}^{\pi}e^{-\mathbf{r}\cos t}\cos(\mathbf{r}\sin t)dt&=\int_{0}^{\pi}e^{-\mathbf{r}_0\cos t}\cos(\mathbf{r}_0\sin t)dt),\\
\int_{0}^{\pi}e^{-\mathbf{r}\cos t}\sin(\mathbf{r}\sin t)dt&=\int_{x_0}^{x}\left(e^{e^s}-e^{-e^s}\right)ds+\int_{0}^{\pi}e^{-\mathbf{r}_0\cos t}\sin(\mathbf{r}_0\sin t)dt.
\end{split}
\end{equation}
Note that $\int_{-\infty}^{x}(e^{e^s}-e^{-e^s})ds=2\int_{0}^{\mathbf{r}}(\sinh t)t^{-1}dt$ converges. By the Cauchy integral theorem, we may let $x_0\to-\infty$ and thus $\mathbf{r}_0\to0$ in \eqref{Eq 5}.

\section{Concluding remarks}\label{Concluding remarks}

In this paper, by proving Theorem~\ref{maintheorem1} and Theorem~\ref{maintheorem2} we provide lower bounds for $|f(z)|$ of the solution of equation \eqref{Eq 1} along certain rays. For a function $f(z)$ which is analytic outside the disc $D(0,\dot{r})$, we define the \emph{order} of growth $\sigma(f)$ of $f(z)$ as
\begin{equation*}
\begin{split}
\sigma(f)=\limsup_{|z|\to\infty}\frac{\log\log|f(z)|}{\log |z|}.
\end{split}
\end{equation*}
If $\sigma(f)=\infty$, then we define the \emph{hyper-order} $\varsigma(f)$ of $f(z)$ as
\begin{equation*}
\begin{split}
\varsigma(f)=\limsup_{r\to\infty}\frac{\log\log\log|f(z)|}{\log |z|}.
\end{split}
\end{equation*}
Now, if $f(z)$ is a meromorphic solution of equation \eqref{Eq 1}, then by Theorem~\ref{maintheorem1} we have $\varsigma(f)\geq \sigma(h)$. On the other hand, using the Wiman--Valiron theory (see~\cite{Hayman1974}), it can be shown that $\varsigma(f)\leq \sigma(h)$~\cite{Zhang2024}. Thus $\varsigma(f)=\sigma(h)$. With this fact, below we look at two problems in uniqueness theory of meromorphic functions and an equation of Hayman mentioned in the introduction, respectively.

First, equation \eqref{Eq 1} is related to a conjecture posed by Br\"{u}ck~\cite{Bruck1996}: \emph{Let $f(z)$ be an entire function which is not constant. If the hyper-order $\varsigma(f)<\infty$ and $\varsigma(f)\not\in \mathbb{N}$, and if $f(z)$ and $f'(z)$ share one value $a$ CM, then $f'(z)-a=c(f(z)-a)$ for some constant $c\in \mathbb{C}$.} Here we say that $f(z)$ and $f'(z)$ share the value $a$ CM if $f(z)-a$ and $f'(z)-a$ have the same zeros, counting multiplicities. Br\"{u}ck himself proved that his conjecture is ture if the shared value $a=0$ or if $f'(z)$ has relatively few zeros in the sense of Nevanlinna theory, i.e.~$N(r,1/f')=o(T(r,f))$, $r\to\infty$. Gundersun and Yang~\cite{GundersenYang1998} proved that Br\"{u}ck's conjecture holds when $f$ has finite order.

Under the assumptions of Br\"{u}ck's conjecture, we have $f'(z)-a=e^{\varphi(z)}(f(z)-a)$, where $\varphi(z)$ is an entire function. When $a\not=0$, if we let $g(z)=(f(z)-a)/a$, then $g(z)$ satisfies the first order differential equation
\begin{equation}\label{Eq diss1}
\begin{split}
g'(z)=e^{\varphi(z)}g(z)+1.
\end{split}
\end{equation}
By using the lemma on the logarithmic derivative and the first main theorem of Nevanlinna, we deduce from \eqref{Eq diss1} that
\begin{equation*}
\begin{split}
T(r,e^{\varphi})=m(r,e^{\varphi})&\leq T(r,a/g)+m(r,g'/g)+O(1)\\
&\leq T(r,g)+O(\log rT(r,g))+O(1)\leq 2T(r,g),
\end{split}
\end{equation*}
where $r\to\infty$ outside a set of finite linear measure. By removing this exceptional set using Borel's lemma~\cite[Lemma~1.1.1]{Laine1993} and then applying the Carath\'{e}odory inequality (see~\cite[pp.~66-67]{YangYi2003}) to $e^{\varphi(z)}$, we may show that $\varphi(z)$ is of finite order. By Theorem~\ref{maintheorem1}, if $\varphi(z)$ is a nonconstant polynomial, then $f(z)$ must be of infinite order and the hyper-order of $f(z)$ is equal to the degree of $\varphi(z)$, contradicting with the assumptions of Br\"{u}ck's conjecture. Thus we have the following

\begin{theorem}\label{maintheorem31}
Let $f(z)$ be a nonconstant entire function such that the hyper-order $\varsigma(f)<\infty$ and $\varsigma(f)\not\in \mathbb{N}$. If $f(z)$ and $f'(z)$ share the value $a$ CM, then $f'(z)-a=e^{\varphi(z)}(f(z)-a)$, where $\varphi(z)$ is a constant or a transcendental entire function of finite order.
\end{theorem}

Second, equation \eqref{Eq 1} can be a particular reduction of a second order differential equation of Hayman~\cite{Hayman1996}, i.e.,
\begin{equation}\label{HAYMANEQ1}
f''(z)f(z)-f'(z)^2+\alpha_1(z)f'(z)f(z)+\alpha_2(z)f(z)^2=\beta_1(z)f(z)+\beta_2(z)f'(z)+\beta_3(z),
\end{equation}
where the coefficients $\alpha_1(z),\alpha_2(z),\beta_1(z),\beta_2(z),\beta_3(z)$ are rational functions. The \emph{transcendental} meromorphic solutions $f(z)$ of \eqref{HAYMANEQ1} satisfy a list of first order or second order linear differential equations~\cite{Zhang2024}. In particular, when $\alpha_1(z)$ and $\alpha_2(z)$ both vanish identically, all transcendental meromorphic solutions of equation~\eqref{HAYMANEQ1} are exponential type functions~\cite{ChiangHalburd2003,HalburdWang}. The order of growth of $f(z)$ can be determined explicitly, except those satisfying the first order differential equation
\begin{equation}\label{Eq 3 a0}
\begin{split}
f'(z)=h(z)f(z)+\gamma_1(z),
\end{split}
\end{equation}
where $\gamma_1(z)$ is a rational function and $h(z)$ is a meromorphic solution of the first order differential equation
\begin{equation}\label{Eq 3 a1}
\begin{split}
h'(z)=\gamma_2(z)h(z)+\gamma_3(z),
\end{split}
\end{equation}
where $\gamma_2(z),\gamma_3(z)$ are rational functions~\cite{zhang2017,Zhang2024}. We may suppose without loss of generality that $\gamma_1(z)$ is identically equal to $0$ or $1$. If $h(z)$ is a rational function, then $\sigma(f)$ is an integer; if $h(z)$ is transcendental and $\gamma_1(z)\equiv0$, then $\varsigma(f)=\sigma(h)$ is an integer~\cite{Zhang2024}. Now, if $h(z)$ is transcendental, $\gamma_1(z)\equiv1$ and $\gamma_3(z)\equiv0$, then by Theorem~\ref{maintheorem1} we have $\varsigma(f)=\sigma(h)$. It is not known if $\varsigma(f)=\sigma(h)$ when $\gamma_1(z)\equiv1$ and $\gamma_3(z)\not\equiv0$.

Here we consider the autonomous case of equations \eqref{Eq 3 a0} and \eqref{Eq 3 a1}, i.e., $f'(z)=h(z)f(z)+c_3$ and $h'(z)=c_1h(z)+c_2$ with three constants $c_1,c_2,c_3$. If $c_1=0$, then $h(z)$ is a constant or a linear polynomial and it follows that $f(z)$ has order $1$ or $2$. If $c_1\not=0$, we may write $h(z)=c_4e^{c_1z}-c_2/c_1$ for a constant $c_4$ and thus $f'(z)=(c_4e^{c_1z}-c_2/c_1)f(z)+c_3$. If $c_4=0$, then $\sigma(f)=1$; if $c_4\not=0$ and $c_3=0$, then $\varsigma(f)=1$. When $c_3c_4\not=0$, by suitable translation and re-scaling, we may suppose that $c_3=c_4=1$, i.e.
\begin{equation}\label{Eq 3 a2}
\begin{split}
f'(z)=(e^{z}-c_5)f(z)+1,
\end{split}
\end{equation}
where $c_5$ is a constant. If $c_5=0$, then by Theorem~\ref{maintheorem1} we have $\varsigma(f)=1$. Below we consider the case when $c_5\not=0$. We write $c_5=a+ib$ with two real constants $a,b$ such that $|a|+|b|\not=0$.

If $c_5$ is a positive integer, say $c_5=p\geq 1$, then it is easy to check that equation \eqref{Eq 3 a2} has a special solution
\begin{equation*}
\begin{split}
f_p(z)=-\sum_{k=1}^{p}\frac{(p-1)!}{(p-k)!}e^{-kz}.
\end{split}
\end{equation*}
Then all other solutions of equation \eqref{Eq 3 a2} take the form $f(z)=f_p(z)+g(z)$, where $g(z)$ is a nontrivial solution of the equation $g(z)'=(e^{z}-c_5)g(z)$. We see that $g(z)$ have hyper-order $\varsigma(g)=1$.

If $c_5$ is not a positive integer, we shall show that $\varsigma(f)=1$ by generalizing the idea in the introduction. For any real constant $\omega$, we consider the complex line integral
\begin{equation}\label{Eq 70 fina-3}
\begin{split}
\oint_{|z|=1}e^{-\omega z}z^{a+ib-1}dz=:C_{\omega}.
\end{split}
\end{equation}
In general, $C_{\omega}$ is dependent on $\omega$. In particular, when $a+ib$ is a negative integer, say $a+ib=-q$, it is elementary that $C_{\omega}= i2\pi(-\omega)^q/(q!)$.

Along the curve $L$ defined in the introduction, by the Cauchy integral theorem, we may write the solution of the equation $f'(z)=(e^{z}-a-ib)f(z)+1$ as
\begin{equation}\label{Eq 70 fina-1}
\begin{split}
f(z)=e^{e^z-(a+ib)z}\left[c+\int_{0}^{x}e^{-e^{t}+at+ibt}dt+i\mathbf{H}\right],
\end{split}
\end{equation}
where $c$ is the integration constant and
\begin{equation}\label{Eq 70 fina-2}
\begin{split}
\mathbf{H}=e^{(a+ib)x}\int_{0}^{y}e^{-e^x\cos t-ie^x\sin t}e^{(a+ib)it}dt.
\end{split}
\end{equation}
By the same arguments as in the derivation in \eqref{Eq 4nalr-1}, we may show that
\begin{equation}\label{Eq 70 fina-3-hy}
\begin{split}
\int_{0}^{x}e^{-e^{t}+at}e^{ibt}dt=\int_{0}^{\infty}e^{-e^{t}+at}e^{ibt}dt-\int_{e^x}^{\infty}e^{-s}s^{a-1}s^{ib}ds=\tilde{d}_0+O\left(e^{ax}e^{-e^x}\right),
\end{split}
\end{equation}
where $\tilde{d}_0$ is a constant. If $c+\tilde{d}_0\not=0$, then we see that $\varsigma(f)=1$ by choosing $y=0$. If $c+\tilde{d}_0=0$, we let $\omega=e^x$. We choose $y=2\mathbf{k}\pi$ for a positive integer $\mathbf{k}$. Then we compare the integrals in \eqref{Eq 70 fina-2} and \eqref{Eq 70 fina-3} and find that
\begin{equation}\label{Eq 70 fina-4-bef}
\begin{split}
i\mathbf{H}=e^{(a+ib)x}C_{\omega}\mathbf{k}=i\frac{2\pi(-1)^q}{q!}\mathbf{k}
\end{split}
\end{equation}
when $a+ib=-q$ is a negative integer, and
\begin{equation}\label{Eq 70 fina-4}
\begin{split}
i\mathbf{H}=e^{(a+ib)x}C_{\omega}\sum_{k=1}^{\mathbf{k}}e^{i 2(k-1)\pi(a+ib)}=e^{(a+ib)x}C_{\omega}\frac{e^{i2\mathbf{k}\pi(a+ib)}-1}{e^{i 2\pi(a+ib)}-1}
\end{split}
\end{equation}
when $a+ib$ is not a nonzero integer. In the latter case, we may choose the integer $\mathbf{k}$ so that $\mathbf{k}(a+ib)$ is not an integer. When $\omega=e^x$ is large, we choose an integer $k$ so that $k+a>0$ and apply the integration by parts $k$ times and find
\begin{equation}\label{Eq 70 fina-6}
\begin{split}
C_{\omega}=\left[\frac{e^{-\omega z}z^{a+ib}}{a+ib}+\cdots+\frac{\omega^ke^{-\omega z}z^{a+k+ib}}{(a+ib)\cdots(a+k+ib)}\right]_{e^{i0}}^{e^{2i\pi}}+W(z),
\end{split}
\end{equation}
where
\begin{equation}\label{Eq 70 fina-6-hg}
\begin{split}
W(z)=\frac{\omega^{k+1}}{{(a+ib)(a+1+ib)\cdots(a+k+ib)}}\oint_{|z|=1}e^{-\omega z}z^{a+k+ib}dz.
\end{split}
\end{equation}
For the integral in \eqref{Eq 70 fina-6-hg}, by the Cauchy integral theorem, it is easy to show that
\begin{equation}\label{Eq 70 fina-7}
\begin{split}
\oint_{|z|=1}e^{-\omega z}z^{a+k+ib}dz&=\left[1-e^{i2\pi(a+k+ib)}\right]\int_{0}^{1}e^{-\omega t}t^{a+k+ib}dt.
\end{split}
\end{equation}
Thus we have
\begin{equation}\label{Eq 70 fina-8}
\begin{split}
\oint_{|z|=1}e^{-\omega z}z^{a+k+ib}dz&=\frac{1-e^{i2\pi(a+k+ib)}}{\omega^{a+k+1+ib}}\left[\int_{0}^{\infty}e^{-s}s^{a+k+ib}ds-\int_{\omega}^{\infty}e^{-s}s^{a+k+ib}ds\right]\\
&=\frac{1-e^{i2\pi(k+a+ib)}}{\omega^{a+k+1+ib}}\left[\Gamma(a+k+1+ib)-\int_{\omega}^{\infty}e^{-s}s^{a+k+ib}ds\right],
\end{split}
\end{equation}
where $\Gamma(z)$ represents the \emph{gamma function of Euler}; see \cite[Chapter~15]{GreeneKrantz2006}. The gamma function of Euler is meromorphic and never vanishes in the plane and, in particular, is analytic in the right half-plane. Then, by the same arguments as in the derivation in \eqref{Eq 4nalr-1}, we obtain from \eqref{Eq 70 fina-6}-\eqref{Eq 70 fina-8} that
\begin{equation}\label{Eq 70 fina-9}
\begin{split}
C_{\omega}=\frac{\Gamma(a+k+1+ib)}{{(a+ib)(a+1+ib)\cdots(a+k+ib)}}\frac{1-e^{i2\pi(a+ib)}}{e^{(a+ib)x}}+O\left(e^{(a+k+1)x}e^{-e^x}\right).
\end{split}
\end{equation}
By \eqref{Eq 70 fina-1}-\eqref{Eq 70 fina-4} and \eqref{Eq 70 fina-9}, we see that, for each integration constant $c$, the solution in \eqref{Eq 70 fina-1} satisfies $|f(z)|\geq \xi e^{e^x}e^{-ax+b2\mathbf{k}\pi}$ for all large $x$ and some constant $\xi$ dependent on $a+ib$ and $\mathbf{k}$. This yields $\varsigma(f)=1$, provided that $a+ib$ is not a positive integer. Together with \cite[Theorem~2.1]{Zhang2024}, this gives a complete description on the order or hyper-order for transcendental entire solutions $f(z)$ of equation \eqref{HAYMANEQ1} in the autonomous case: \emph{If $f(z)$ has finite order, then $\sigma(f)=1$ or $\sigma(f)=2$ and, if $f(z)$ has infinite order, then $\varsigma(f)=1$}.

Finally, we give some more comments on equation \eqref{Eq 3 a2}. If $c_5=0$, we have equation \eqref{Eq 1} and the solution $f(z)$ in \eqref{Eq 4} takes the form
\begin{equation*}
\begin{split}
f(z)=e^{e^x\cos y}e^{ie^x\sin y}\left[c+\int_{x_0}^xe^{-e^t}dt+i\mathbf{H}\right],
\end{split}
\end{equation*}
where $\mathbf{H}$ is defined in \eqref{Eq 70-cbuy}. Recall that $\int_{x_0}^xe^{-e^t}dt=d_0+O\left(xe^{-e^x}\right)$ for a constant $d_0$. Along the line $\Upsilon_{\mathbf{k},1}$ on which $z=x+i(2\mathbf{k}\pi+\zeta)$ such that $\zeta\in[-\pi/2,\pi/2]$, the calculations in the introduction actually shows that $|\mathbf{H}-2\mathbf{k}\pi|=|\mathbf{I}_{12}+i\mathbf{I}_{22}|\leq 2|\zeta|e^{-e^x\cos\zeta}$. We may choose the integer $\mathbf{k}$ so that $c+d_0+i2\mathbf{k}\pi$ is nonzero and thus $|f(z)|\geq (1-\varepsilon)|c+d_0+i2\mathbf{k}\pi|e^{e^x\cos\zeta}$. On the other hand, along the line $\Upsilon_{\mathbf{k},2}$ on which $z=x+i(2\mathbf{k}\pi+\zeta)$ such that $\zeta\in[\pi/2+\varepsilon,3\pi/2-\varepsilon]$, by Theorem~\ref{maintheorem2} we have $|f(z)|\geq \xi/2e^{x}\cos(\zeta-\pi)$. This implies that all zeros, except finitely many, of the solution $f(z)$ in \eqref{Eq 4} in the right half-plane must be around the lines $\Upsilon_{k,3}$ on which $z=x+i(2k+1)\pi/2$, $k=0,\pm 1,\pm2,\cdots$. It is an easy work to extend this result to the left half-plane. If $c_5\not=0$, then by equation \eqref{Eq 70 fina-3}, we have the two identities
\begin{equation*}
\begin{split}
\int_{0}^{2\pi}e^{-\omega\cos\theta-b\theta}\cos(-\omega\sin\theta+a\theta)d\theta&=C_{\omega,1},\\
\int_{0}^{2\pi}e^{-\omega\cos\theta-b\theta}\sin(-\omega\sin\theta+a\theta)d\theta&=C_{\omega,2},
\end{split}
\end{equation*}
where $C_{\omega,1}$ and $C_{\omega,2}$ are two real constants. This generalizes the two identities in \eqref{Eq 70-cor4}. With these two identities, we may estimate the modulus $|f(z)|$ of the solution $f(z)$ of equation \eqref{Eq 3 a2} along lines parallel to the real axis as above.

Further, we take the derivatives on both sides of equation \eqref{Eq 3 a2} and obtain the second order differential equation $f''(z)-(e^{z}-c_5)f'(z)-e^{z}f(z)=0$. If we let $f(z)=w(z)\exp(\frac{1}{2}(e^z-c_5z))$, then $w(z)$ satisfies the second order differential equation
\begin{equation}\label{Eq 3 a3}
\begin{split}
w''(z)+\left[-\frac{1}{4}e^{2z}+\frac{1}{2}(c_5-1)e^z-\frac{1}{4}c_5^2\right]w(z)=0.
\end{split}
\end{equation}
The method by Bank and Langley~\cite{Banklangley1987} (see also~\cite[Chapter~5]{Laine1993}) can be used to detect whether $w(z)$ has relatively few zeros in the sense that the exponent of convergence of zeros of $w(z)$ satisfies $\lambda(w)<\infty$~\cite{zhang2021-1,Zhang2022}. These results have applications in describing the oscillating phenomenon of the general second order linear differential equation~\cite{Bergweilereremenko2017,Bergweilereremenko2019,Zhang2024-Construction1,Zhang2025-Construction2}. We note that equation \eqref{Eq 3 a3} has a zero-free solution $w(z)=\exp(\frac{1}{2}(e^z-c_5z))$. The other solutions $w(z)$ of \eqref{Eq 3 a3} such that $0<\lambda(w)<\infty$ can be determined explicitly; see~\cite{ChiangIsmail2006,Zhang2022}. By the expression of $f(z)$ in \eqref{Eq 70 fina-1}, the method presented previously allows to describe the modulus of solutions of equation \eqref{Eq 3 a3} more precisely.




\end{document}